\documentclass[11pt]{article}

\title{Reduction of multisymplectic manifolds}
\author{Casey Blacker}\date{}

\usepackage{geometry}
\usepackage{tikz}
\usepackage{cite}
\usepackage[utf8]{inputenc}
\usepackage[shortlabels]{enumitem}
\usepackage{amsmath,amsfonts,amssymb,amsthm}
\usepackage{comment}
\usepackage[hidelinks]{hyperref}

\theoremstyle{plain}
\newtheorem{theorem}{Theorem}[section]
\newtheorem*{theorem*}{Theorem}
\newtheorem{lemma}[theorem]{Lemma}
\newtheorem{proposition}[theorem]{Proposition}

\theoremstyle{definition}
\newtheorem{definition}[theorem]{Definition}
\newtheorem{example}[theorem]{Example}

\theoremstyle{remark}
\newtheorem{remark}{Remark}[section]

\let\bibtexi\i

\newcommand{\Z}{\mathbb{Z}}

\newcommand{\R}{\mathbb{R}}
\newcommand{\C}{\mathbb{C}}

\renewcommand{\i}{\mathrm{i}}
\newcommand{\g}{\mathfrak{g}}
\renewcommand{\t}{\mathfrak{t}}

\newcommand{\Hom}{\mathrm{Hom}}

\newcommand{\Ad}{\mathrm{Ad}}

\newcommand{\F}{\mathcal{F}}
\renewcommand{\L}{\mathcal{L}}

\newcommand{\X}{\mathfrak{X}}

\renewcommand{\subset}{\subseteq}

\renewcommand{\d}{\mathrm{d}}

\begin{document}\maketitle

\begin{abstract}
	We extend the Marsden--Weinstein--Meyer symplectic reduction theorem to the setting of multisymplectic manifolds. In this context, we investigate the dependence of the reduced space on the reduction parameters. With respect to a distinguished class of multisymplectic moment maps, an exact stationary phase approximation and nonabelian localization theorem are also obtained.
\end{abstract}

\let\thefootnote\relax
\footnote{\textit{Date.} January 28, 2021}
\footnote{\textit{2020 Mathematics Subject Classification.} 53D05, 53D20, 70S05, 70S10}
\footnote{\textit{Key words and phrases.} multisymplectic geometry, moment maps, Duistermaat--Heckman theorems.}

\tableofcontents


\section{Introduction}

A $k$-plectic structure on a smooth manifold $M$ is a closed $(k+1)$-form $\omega\in\Omega^{k+1}(M)$ which is nondegenerate in the sense that the assignment $X\mapsto\iota_X\omega$ defines an inclusion of vector bundles $TM\hookrightarrow\Lambda^kT^*M$. This extends the familiar construction of a symplectic structure on $M$, constituting the case $k=1$, to the setting of higher degree forms. As symplectic geometry forms the language of classical mechanics, so multisymplectic geometry provides a framework for classical field theories \cite{RyvkinWurzbacher19,Roman-Roy09,deLeondeDiegoSantamariaMerino03}. We may thus characterize multisymplectic manifolds, together with the attendant dynamical formalism, as the mathematical extensions of physical classical field theories.

Recall that, while the cotangent bundle $T^*M$ of a smooth manifold $M$ exhibits a canonical symplectic structure \cite{Marsden92}, it is not always the case that $M$ itself admits a symplectic form. In contrast,  when $\dim M\geq 7$ and $3\leq k\leq \dim M-2$, the $k$-plectic structures on $M$ are generic in the space of closed $(k+1)$-forms \cite[Theorem 2.2]{Martinet70} (see also the discussion around \cite[Theorem 3.11]{RyvkinWurzbacher19}). It is thus natural to anticipate substantially weaker results in the general multisymplectic setting as compared with that of the symplectic. Indeed, this will prove to be the case with regard to the failure of the multisymplectic reduction procedure to ensure the nondegeneracy of the reduced form, as described below.

Before outlining the structure of this paper, we first recall the reduction theorem in the original symplectic setting. We refer to \cite[Chapter 2]{Marsden92} for a more thorough review of this material according to the motivating physical perspective.

The smooth functions $f\in C^\infty(M)$ on a symplectic manifold $(M,\omega)$ individually encode infinitesimal symmetries $X\in\X(M)$, $\L_X\omega=0$, according to the relation $\d f=\iota_X\omega$. When a family of smooth functions $(\mu_\xi)_{\xi\in\g}$, linear in $\xi\in\g$, determines in this manner the fundamental vector fields $\underline\xi\in\X(M)$ of the action of a Lie group $G$ on $(M,\omega)$, the action of $G$ is said to be a (weakly) \emph{Hamiltonian action} and the assignment $\mu:M\to\g^*$, encoding the functions $(\mu_\xi)_{\xi\in\g}$ through pointwise contraction $\mu_\xi=\langle\mu,\xi\rangle$, is called an associated (weak) \emph{moment map}. We may remove the designation ``weak'' by imposing a Lie algebra homomorphism condition on the assignment $\xi\mapsto\mu_\xi$ with respect to a natural Lie bracket on the space of smooth functions $C^\infty(M)$ associated to the symplectic structure $\omega$.

The symplectic reduction theorem, due independently to Marsden--Weinstein \cite{MarsdenWeinstein74} and Meyer \cite{Meyer73}, exploits this interaction between functions and symmetries to systematically remove certain degrees of freedom on a symplectic manifold $(M,\omega)$ in the presence of a Hamiltonian action of $G$ and an associated moment map $\mu:M\to\g^*$. Here we present the formulation of Marsden and Weinstein.

\begin{theorem*}[Marsden--Weinstein]
	Let $G$ be a Lie group acting symplectically on the symplectic manifold $(M,\omega)$. Let $\mu:M\to\g^*$ be a moment map for the action, and $\lambda\in\g^*$ a regular value of $\lambda$. Suppose that $G_\lambda$ acts freely and properly on the manifold $\mu^{-1}(\lambda)$. Then if $i_\lambda:\mu^{-1}(\lambda)\to M$ is the inclusion, there is a unique symplectic structure $\omega_\lambda$ on the reduced phase space $M_\lambda$ such that $\pi_\lambda^*\omega_\lambda=i_\lambda^*\omega$ where $\pi_\lambda$ is the projection of $\mu^{-1}(\lambda)$ onto $M_\lambda$.
\end{theorem*}

The symplectic reduction theorem has been adapted to the settings of contact structures \cite{deLeonLainzValcazar19,Willett02}, cosymplectic manifolds \cite{Albert89}, polysymplectic manifolds \cite{MarreroRoman-RoySalgadoVilarino15}, certain higher Poisson structures\cite{BursztynMartinezAlbaRubio19}, Courant algebroids\cite{BursztynCavalcantiGualtieri07}, and quasi-Hamiltonian $G$-spaces\cite{AlekseevMalkinMeinrenken98}. Reduction schemes have also appeared in the multisymplectic approach to classical field theory \cite{Sniatycki04,MarsdenMontgomeryMorrisonThompson86,CastrillonLopezRatiuShkoller00}. Reduction of general multisymplectic manifolds is discussed in \cite[p.\ xxiv]{OrtegaRatiu04} and \cite{MarsdenWeinstein01}, and the topic is treated in depth in \cite{Echeverria-EnriquezMunoz-LecandaRoman-Roy18}.

The aim of this paper is to extend the Marsden--Weinstein--Meyers symplectic reduction theorem to the multisymplectic setting, and to investigate the dependence of the reduced space on the reduction parameters.

We begin in Section \ref{sec:manifolds} with a review of the theory of multisymplectic manifolds, broadly following the exposition of Ryvkin and Wurzbacher \cite{RyvkinWurzbacher19} and the earlier paper of Cantrijn, Ibort, and de Le{\'o}n \cite{CantrijnIbortdeLeon99}. The main idea for our purposes is that certain $(k-1)$-forms $\alpha\in\Omega^{k-1}(M)$ determine multisymplectic symmetries $X\in\X(M)$, $\L_X\omega=0$, in a manner analogous to the symplectic case: namely, $\d\alpha=\iota_X\omega$.

A key difference between our treatment and the usual conventions is that we define the $k$-plectic \emph{bracket} on the space of Hamiltonian $(k-1)$-forms $\Omega^{k-1}(M)$ by means of the Lie derivative, $\{\alpha,\beta\}=\L_{X_\alpha}\beta$, instead of the interior derivative, $\{\alpha,\beta\}'=\iota_{X_\alpha}\d\beta$. The difference, an exact term $\d\iota_{X_\alpha}\beta$, vanishes in the symplectic case. With respect to our conventions, the space of Hamiltonian forms $\Omega_H^{k-1}(M)$ inherits the structure of a \emph{Leibniz algebra}, also called a \emph{Loday algebra}, which are defined in \cite{LodayPirashvili93} to consist of a vector space $V$ and a bilinear mapping $[\,,]:V\times V\to V$ satisfying the Jacobi identity but not necessarily antisymmetric. The usual multisymplectic bracket $\{\,,\}'$, on the other hand, is bilinear and antisymmetric, but does not satisfy the Jacobi identity. Both constructions $\{\,,\}$ and $\{\,,\}'$ descend to identical Lie brackets on the quotient of the space of Hamiltonian forms by the exact forms $\Omega_H^{k-1}(M)/\mathrm{im}\,\d_{k-2}$.

In Section \ref{sec:Hamiltonian_systems} we present a theory of multisymplectic moment maps and Hamiltonian actions. Our construction is similar to the \emph{covariant momentum map} of classical field theory \cite{GotayIsenbergMarsdenMontgomery1999} and is broadly consistent with the framework presented by Echeverr\'{\bibtexi}a-Enr\'{\bibtexi}quez, Mu\~{n}oz-Lecanda, and Rom\'{a}n-Roy \cite{Echeverria-EnriquezMunoz-LecandaRoman-Roy18}. We consider the $k$-plectic moment map as a $\g^*$-valued differential form $\mu\in\Omega^{k-1}(M,\g^*)$ encoding a $\g$-parameterized family of Hamiltonian $(k-1)$-forms $\mu_\xi=\langle\mu,\xi\rangle\in\Omega^{k-1}(M)$ which collectively generate the action of a group of symmetries $G$, subject to a certain Leibniz algebra compatibility condition.

Here we introduce the particularly tractable class of \emph{split moment maps}: those that factor as $\mu=\nu\wedge\eta\in\Omega^{k-1}(M)$, where $\nu:M\to\g^*$ is a smooth function and $\eta\in\Omega^{k-1}(M)$ is a closed form. Informally, we consider $\nu$ as analogous to a symplectic $\g^*$-valued moment map, and view $\eta$ as an auxiliary form. The justification for this construction is twofold: first, it provides a multisymplectic environment in which the resemblance with the symplectic formalism is more transparent; second, imposes a minimal condition under which multisymplectic $G$-spaces may be studied by means of a vector-valued map $\nu:M\to\g^*$. Indeed, this class of moment maps is both convenient to work with and provides an approach by which the statements of many classical results of symplectic geometry, in terms of $\g^*$-valued functions, may be interpreted in the multisymplectic context.

In Section \ref{sec:reduction} we present our first main result:

\setcounter{section}{4}
\setcounter{theorem}{0}
\begin{theorem}[Multisymplectic reduction]
	Let $(M,\omega,G,\mu)$ be a $k$-plectic Hamiltonian $G$-space with moment map $\mu$, let $\phi\in\Omega^{k-1}(M,\g^*)$ be a closed form, and let $M_\phi= \mu^{-1}(\phi)/G_\phi$. If $\mu^{-1}(\phi)\subset M$ is an embedded submanifold and $G$ acts freely on $\mu^{-1}(\phi)$, then there is a unique, closed $\omega_\phi\in\Omega^{k+1}(M_\phi)$ satisfying $i^*\omega = \pi^*\omega_\phi$, where $i:\mu^{-1}(\phi)\to M$ is the inclusion and $\pi:\mu^{-1}(\phi)\to M_\phi$ is the quotient map.
	\begin{center}
		\begin{tikzpicture}
			\node (A) at (0,0) {$\mu^{-1}(\phi)$};
			\node (B) at (1.8,0) {$M$};
			\node (C) at (0,-1.8) {$M_\phi$};

			\draw[->] (A) to node[above] {$i$} (B);
			\draw[->] (A) to node[left] {$\pi$} (C);
		\end{tikzpicture}
	\end{center}
\end{theorem}

In contrast to the symplectic case, the reduced form $\omega_\phi\in\Omega^{k-1}(M_\phi)$ may be degenerate. This property is descriptive of the polysymplectic case as well \cite{MarreroRoman-RoySalgadoVilarino15}. Additionally, we describe the reduction of forms and vector fields in Theorem \ref{thm:reduction_of_dynamics}, and obtain some further results on the split Hamiltonian case in Proposition \ref{prop:split_moment_map_free_action}.

In Section \ref{sec:DH} we investigate the dependence of the reduced space $(M_\phi,\omega_\phi)$ on the closed form $\phi\in\Omega(M,\g^*)$. Our guiding model is the foundational result of Duistermaat and Heckman \cite{DuistermaatHeckman82}.

\begin{theorem*}[Duistermaat--Heckman]
	Suppose that $(M,\omega)$ is a symplectic manifold equipped with the Hamiltonian action of a torus $T$, and that $\mu:M\to\t^*$ is an associated moment map. If $\lambda,\lambda_0\in\mathfrak{t}^*$ lie in the same connected component $C$ of the set of regular values of $\mu$, then
	\[
		[\omega_\lambda] = [\omega_{\lambda_0}] + \langle c,\lambda-\lambda_0\rangle
	\]
	where $c\in H^2(M_\lambda,\t)$ denotes the (common) Chern class of the fibrations $q_\lambda:\mu^{-1}(\lambda)\to M_\lambda$, $\lambda\in C$, and we have used the canonical identification of the $H^2(M_\lambda,\R)$ along any $\lambda$-path in $C$ from $\lambda_0$ to $\lambda$.
\end{theorem*}

Specifically, our aim is to describe the variation of the reduced form $\omega_\phi\in\Omega^{k+1}(M_\phi)$ with respect to infinitesimal variations of the closed parameter $\phi\in\Omega^{k-1}(M,\g^*)$, subject to certain assumptions on the structure of $(M,\omega)$ in the vicinity of $\mu^{-1}(\phi)$. To this end, we introduce the notion of \emph{conjugate distributions} on a multisymplectic manifold, in terms of which we present our second main result:

\setcounter{section}{5}
\setcounter{theorem}{5}
\begin{theorem}[Variation of the multisymplectic reduced space]
	Let $T$ be a torus, let $(M,\omega,T,\mu)$ be a $k$-plectic Hamiltonian $T$-space, fix a $T$-invariant closed form $\phi\in\Omega^{k-1}(M,\t^*)$ such that $\mu^{-1}(\phi)\subset M$ is an embedded submanifold on which $T$ acts freely, choose an open subset $C\subset\t^*$, let $\eta\in\Omega^{k-1}(M)$ be $T$-invariant, and write $P=C\wedge\eta+\phi$. If
	\begin{enumerate}[i.]
		\item the diagram
			\begin{center}
				\begin{tikzpicture}
					\node (A) at (0,0) {$\mu^{-1}(P)$};
					\node (B) at (2.5,0) {$\mu^{-1}(\phi) \times P$};
					\node (C) at (0,-1.8) {$P$};

					\draw[->] (A) to node[above] {$\sim$} (B);
					\draw[->] (A) to node[left] {$\mu$} (C);
					\draw[->] (B) to node[below right] {$\pi_2$} (C);
				\end{tikzpicture}
			\end{center}
			is a trivialization of a family of $T$-principal bundles modeled on $\mu^{-1}(\phi)$, and
		\item the fundamental distribution $\underline\t$ is strongly conjugate to a distribution $\underline\t^*\subset TM$ with respect to $\eta$,
	\end{enumerate}
	then, 
	\[
		\partial_\lambda\, [\omega_\psi] = \langle c,\lambda\rangle \wedge [\eta_\psi],		\hspace{1.5cm}\lambda\in C,\;\psi\in P
	\]
	where $c\in\Omega^2(M_\phi,\t)$ is the Chern class of the model space $\mu^{-1}(\phi)\to M_\phi$.
\end{theorem}

In Section \ref{sec:localization} we return to the setting of split moment maps and exhibit an exact stationary phase approximation and a nonabelian localization theorem. The statements and proofs follow easily by emulating the symplectic case. Such an approach may prove fruitful in extending further symplectic results to the multisymplectic setting.

We conclude in Section \ref{sec:outlook} with some ideas for further development. Specifically, we consider stronger and weaker reduction theorems, further analysis of split moment maps, interactions with more sophisticated multisymplectic realizations of the symplectic moment map, applications to infinite-dimensional symplectic manifolds, and multisymplectic quantization.

Throughout the text we frequently invoke a closed form $\eta\in\Omega^{k-1}(M)$. In every instance the symplectic case is retrieved by taking $k=1$ and $\eta=1\in\Omega^0(M)$.

\vspace{.5cm}\noindent
\textbf{Notation and Conventions.} All manifolds $M$ are assumed to be $C^\infty$ and all Lie groups $G$ are assumed to be compact. In particular, all actions are proper. Given an action of $G$ on $M$, we define the fundamental vector field associated to $\xi\in\g$ at each point $x\in M$ by $\underline\xi_x  = \frac{\d}{\d t}\hspace{1pt}e^{-t\xi}\hspace{1pt} x \hspace{1pt}\big|_{t=0}$, so that $\xi\mapsto\underline\xi$ is a Lie algebra homomorphism. We denote the interior product and Lie derivative with respect to $\underline\xi\in\X(M)$ by $\iota_\xi$ and $\L_\xi$, respectively, and adopt a similar convention in Sections \ref{sec:DH} and \ref{sec:localization} with respect to elements $\lambda,\tau\in\g^*$. We denote the natural pairing on $\g^*\otimes\g$ by $\langle\,,\rangle$. For $\phi\in\Omega^*(M,\g^*)$ and $\xi\in\g$ we write $\phi_\xi$ for the pointwise contraction $\langle\phi,\xi\rangle\in\Omega^*(M)$. Thus if $\mu\in\Omega^{k-1}(M,\g^*)$ is a $k$-plectic moment map, then $\tilde\mu:\xi\mapsto\mu_\xi$ is the associated comoment map. A Hamiltonian vector field $X\in\X(M)$ is associated to a Hamiltonian form $\alpha\in\Omega_H^{k-1}(M)$ by the relation $\d\alpha=\iota_X\omega$, and the bracket is defined by $\{\alpha,\beta\}=\L_{X_\alpha}\beta$. In the symplectic case, this is $\{f,h\}=-\omega(X_f,X_h)$. 

\setcounter{section}{1}
\setcounter{theorem}{0}


\section{Multisymplectic manifolds}\label{sec:manifolds}

In this section we present the basic elements of multisymplectic manifolds. We refer to \cite{CantrijnIbortdeLeon99,RyvkinWurzbacher19} for further background on multisymplectic manifolds and to \cite{ForgerPauflerRomer03} for multivector calculus.

Let $M$ be a smooth manifold.

\begin{definition}
	A $(k+1)$-form $\omega\in\Omega^{k+1}(M)$ is said to be a \emph{$k$-plectic structure} on $M$ if it is closed and nondegenerate, in the sense that the map
	\begin{align*}
		\iota\,\omega:	TM	&\to	\Lambda^k\, T^*\! M	\\
				X	&\mapsto	\iota_X\omega
	\end{align*}
	is an inclusion of vector bundles on $M$. A \emph{multisymplectic structure} is a $k$-plectic structure for some $k\geq 1$. 
\end{definition}

If $\omega$ is only known to be closed then we say that $\omega$ is a \emph{premultisymplectic structure} on $M$.

\begin{example}\label{eg:standard_examples_manifolds}
	\begin{enumerate}[i.]
		\item If $(M^{2n},\sigma)$ is a symplectic manifold, then $\sigma^\ell$ is a $(2\ell-1)$-plectic structure on $M$ for $1\leq\ell\leq n$. This class of multisymplectic manifold is investigated in \cite{BarronShafiee19}. 
		\item Let $G$ be a semisimple Lie group and $\langle\,,\rangle$ an $\Ad$-invariant metric on $\g$. For example, we may take $\langle\,,\rangle$ to be the Killing metric. The unique bi-invariant form $\omega\in\Omega^3(G)$ satisfying
			\[
				\omega(X,Y,Z) = \big\langle[X,Y],Z\big\rangle,	\hspace{.8cm}X,Y,Z\in T_1G\cong\g
			\]
			at $1\in G$, is a $2$-plectic structure on $G$. Equivalently, we define
			\[
				\omega = -\langle\d\theta,\theta\rangle = -\langle\d\bar\theta,\bar\theta\rangle,
			\]
			where $\theta,\bar\theta\in\Omega^1(G,\g)$ are the left and right Maurer--Cartan forms on $G$, respectively.
		\item Let $\pi_E:E\to\Sigma$ be a smooth fiber bundle. We will say that an element $\gamma\in\Lambda^kT_x^*E$ is \emph{$\ell$-semihorizontal} if
			\[
				\iota_{v_1}\ldots\iota_{v_{\ell+1}}\gamma = 0
			\]
			for all vertical tangent vectors $v_1,\ldots,v_{\ell+1} \in \ker\d(\pi_E)_x$, and denote by $\pi:\Lambda_\ell^k T^* E\to E$ the bundle of $\ell$-semihorizontal $k$-forms on $E$. The \emph{canonical $k$-form} $\theta\in\Omega^k(\Lambda_\ell^k\hspace{1pt}T^*E)$ is given by
			\[
				\theta_\gamma(X_1,\ldots,X_k) = \gamma(\pi_*X_1,\ldots,\pi_*X_k),	\hspace{.8cm}X_i\in T_\gamma(\Lambda_\ell^k\hspace{1pt}T^*E),
			\]
			and the \emph{canonical $k$-plectic structure} on $\Lambda_\ell^kT^*E$ is defined to be $-\d\theta\in\Omega^k(\Lambda_\ell^k\hspace{1pt}T^*E)$.

			The space of semihorizontal $k$-forms $\Lambda_1^k\hspace{1pt}T^*E$ plays a foundational role in certain approaches to classical field theories \cite{deLeondeDiegoSantamariaMerino03,Roman-Roy09} and in this respect may be considered to represent the fundamental example of a multisymplectic manifold.
	\end{enumerate}
\end{example}

See \cite[Section 3]{RyvkinWurzbacher19} for an abundance of further examples.

Recall that a \emph{$k$-multivector field} on $M$ is a section $X\in\X^k(M)$ of the $k$th exterior power of the tangent bundle $\Lambda^k TM$. The \emph{Schouten bracket} $[\,,]$ is the extension of the Lie bracket from $\X(M)$ to a bilinear form on $\X^*(M)=\Lambda^*TM$ according to the rule that
\[
	[X_1\wedge\ldots\wedge X_k,\,Y_1\wedge\ldots\wedge Y_\ell] = \sum_{i,j} \,(-1)^{i+j}\,[X_i,Y_j]\wedge \hat{X}_i\wedge \hat{Y}_j
\]
where $\hat{X}_i$ designates the expression $X_1\wedge\ldots\wedge X_k$ with the omission of the $i$th term. The interior product and Lie derivative extend to this setting according to the rules
\[
	\iota_X = \iota_{X_k}\ldots\iota_{X_1}
\]
and
\[
	\L_X = \d\iota_X - (-1)^k \,\iota_X\d.
\]
See \cite[Appendix A]{ForgerPauflerRomer03} for more information on the Schouten bracket and for various identities of multivector field operations.

\begin{definition}
	A smooth transformation $\phi:M\to M$ is called a \emph{multisymplectomorphism}, or a \emph{symmetry}, of $(M,\omega)$ if $\phi^*\omega = \omega$. The action of a Lie group $G$ on $M$ is said to be a \emph{multisymplectic action} if it acts by multisymplectomorphisms. A vector field $X\in\X(M)$ is called a \emph{multisymplectic vector field}, or an \emph{infinitesimal symmetry}, if $\L_X\omega=0$.
\end{definition}

If $X\in\X(M)$ is multisymplectic, then the closedness of $\omega$ implies $\d\iota_X\omega = \L_X\omega = 0$, so that $\iota_X\omega$ is locally exact. The case of global exactness is distinguished by the following definition.

\begin{definition}
	Let $\alpha\in\Omega^{k-\ell}(M)$ and $X\in\X^\ell(M)$ for some $\ell\leq k$. If $\d\alpha=\iota_X\omega$, then we say that $\alpha$ is a \emph{Hamiltonian form} for $X$, and that $X$ is a \emph{Hamiltonian multivector field} for $\alpha$.
\end{definition}

We will denote by $\Omega_H^{k-\ell}(M)$ the space of Hamiltonian $(k-\ell)$-forms. Less frequently, we write $\X_H^\ell(M)$ for the space of Hamiltonian $\ell$-vector fields on $M$.

\begin{remark}
\begin{enumerate}
	\item The $(k-1)$-form $\alpha\in\Omega^{k-1}(M)$ is Hamiltonian if and only if $\d\alpha_x$ lies in the image of the inclusion $\iota\,\omega_x:T_xM\hookrightarrow\Lambda^{k-1}T_x^*M$ for every $x\in M$. In this situation, by the nondegeneracy of $\iota\,\omega$, the vector field $X\in\X(M)$ associated to $\alpha$ is uniquely defined, and the preceding discussion shows that $X$ is an infinitesimal symmetry of $(M,\omega)$. The Hamiltonian forms associated to $X$ consists precisely of $\alpha+\beta\in\Omega^{k-1}(M)$ for $\d\beta=0$.
	\item In the symplectic setting, every function is a Hamiltonian $0$-form. For $\ell=1$, a comparison of the rank of $TM$ with that of $\Lambda^k T^*M$ shows this property to be unique to the cases $k=1$ and $k=n-1$, where $n$ is the dimension of $M$.
\end{enumerate}
\end{remark}

\begin{proposition}\label{prop:vanishing_vector_field}
	If $M$ is compact and $\alpha\in\Omega_H^{k-1}(M)$ is proportional to a closed form $\eta\in\Omega^{k-1}(M)$, so that $\alpha = f\eta$ for some $f\in C^\infty(M)$, then the vanishing set of the associated $X\in\X_H(M)$ is nonempty.
\end{proposition}

\begin{proof}
	The nondegeneracy of $\omega$ and the identity
	\[
		\iota_X \omega = \d\alpha = \d f\wedge\eta
	\]
	together imply that the vanishing set of $X\in \X_H(M)$ includes the critical set of $f$, which is nonempty by the compactness of $M$.
\end{proof}

If $h\in C^\infty(M)$ is a function on a symplectic manifold $(M,\omega)$ then we may consider the triple $(M,\omega,h)$ as an \emph{abstract mechanical system}. The \emph{dynamics}, or \emph{equations of motion}, of $(M,\omega,h)$ is the associated Hamiltonian vector field $X_h\in\X_H(M)$. A \emph{Hamiltonian curve}, or \emph{solution of the equations of motion}, is an immersion $\phi:\R\to M$ such that $(\partial_t\phi)(s)=X_h(\phi(s))$ for all $s\in\R$. Adapting the terminology of \cite{Smale70}, if $G$ acts symplectically on $(M,\omega)$ preserving $h$ then we consider $(M,\omega,G,h)$ to be an \emph{abstract mechanical system with symmetry}. The consideration of the field theoretic case motivates the following definition.

\begin{definition}
	An \emph{abstract field theory} $(M,\omega,h)$ consists of a $k$-plectic manifold $(M,\omega)$ and a Hamiltonian function $h\in C_H^\infty(M)$. The \emph{equations of motion} take the form of a $k$-vector field $X_h\in\X_H^k(M)$. A \emph{Hamiltonian $k$-curve} is an immersion $\phi:\Sigma\to M$ of a $k$-dimensional manifold $\Sigma$ with distinguished section $\partial_\Sigma\in\X^k(\Sigma)$ such that $(\phi_*\partial_\Sigma)(x) = X_h(\phi(x))$ for all $x\in\Sigma$. If the Lie group $G$ acts multisymplectically on $(M,\omega)$ and preserves $h$ then we call $(M,\omega,G,h)$ an \emph{abstract field theory with symmetry}.
\end{definition}

We refer to \cite[Section 2.3]{RyvkinWurzbacher19} and \cite{Helein12} for physical origins of this terminology.

\begin{definition}
	We define the \emph{bracket} $\{\,,\}$ on $\Omega_H^{k-1}(M)$ to be the bilinear map
	\[
		\{\alpha,\beta\} = \L_{X_\alpha}\,\beta
	\]
	for $\alpha,\beta\in\Omega_H^{k-1}(M)$.
\end{definition}

\begin{remark}
	\begin{enumerate}
		\item Note that the bracket is well-defined since the associated Hamiltonian vector field $X_\alpha$ is unique. 
		\item Our definition of $\{\,,\}$ differs from the usual assignment $(\alpha,\beta)\mapsto\iota_{X_\alpha}\d\beta$, which is antisymmetric but which satisfies the Jacobi identity only up to the addition of exact terms. Our choice of bracket will be shown to satisfy the Jacobi identity but is antisymmetric only up to exact terms, as
		\[
			\{\alpha,\beta\} + \{\beta,\alpha\} = \d(\iota_{X_\alpha}\beta + \iota_{X_\beta}\alpha).
		\]
		Both constructions coincide with the Poisson bracket in the symplectic case. See \cite{Rogers12} for a comparison of the two and \cite{CantrijnIbortdeLeon96} for a similar construction.
	\end{enumerate}
\end{remark}

\begin{definition}
	A \emph{Leibniz algebra} $(V,[\,,])$ consists of a vector space $V$ and a bilinear map $[\,,]:V\times V\to V$ which satisfies the Jacobi identity.
\end{definition}

\begin{lemma}
	The bracket $\{\,,\}$ preserves $\Omega_H^{k-1}(M)$ and endows it with the structure of a Leibniz algebra. Moreover, the pushforward of $\{\,,\}$ under the assignment of Hamiltonian vector fields is the standard Lie bracket on $\X(M)$.
\end{lemma}

\begin{proof}
	We obtain $X_{\{\alpha,\beta\}} = [X_\alpha,X_\beta]$ from
	\[
		\d\{\alpha,\beta\} = \L_{X_\alpha}\iota_{X_\beta}\omega = \iota_{[X_\alpha,X_\beta]}\omega,
	\]
	and the Jacobi identity follows as $[\L_{X_\alpha},\L_{X_\beta}] \gamma = \L_{X_{\{\alpha,\beta\}}}\gamma$.
\end{proof}

\begin{proposition}\label{prop:proportionally_closed_Hamiltonian_forms}
	If a closed form $\eta\in\Omega^{k-1}(M)$ divides $\omega$, so that $\omega=\sigma\wedge\eta$ for some $\sigma\in\Omega^2(M)$, and if $j:L\to M$ is an integral manifold of the kernel distribution $\F\subset TM$ of $\eta$, then
	\begin{enumerate}[i.]
		\item $j^*\sigma\in\Omega^2(M)$ is a presymplectic structure on $L$ and is independent of the choice of $\sigma$,
		\item if $f\eta\in\Omega_H^{k-1}(M)$ with associated vector field $X\in\X(M)$ then
			\[
				\d j^*f = j^*\iota_X\sigma,
			\]
		\item if $X$ is tangent to $L$, then $X|_L\in\X(L)$ is the Hamiltonian vector field associated to $j^*f\in C^\infty(L)$ with respect to the presymplectic structure $j^*\sigma\in\Omega^2(L)$.
	\end{enumerate}
\end{proposition}

\begin{proof}
	\begin{enumerate}[i.]
		\item Fix $x\in L$. From the nondegeneracy of $\omega$ we deduce that $\eta$ is nonvanishing and it follows from
		\[
			\d\sigma(u,v,w)\wedge\eta = \iota_w\iota_v\iota_u \d\omega = 0,\hspace{.8cm}u,v,w\in \F_x
		\]
		that $j^*\sigma$ is closed. Similarly, if $\sigma'\in\Omega^2(M)$ with $\omega=\sigma'\wedge\eta$ then the identity
		\[
			(\sigma-\sigma')(u,v)\wedge\eta = \iota_v\iota_u(\sigma\wedge\eta-\sigma'\wedge\eta) = 0,\hspace{.8cm}u,v\in \F_x
		\]
		implies that $j^*\sigma$ is independent of $\sigma$.

		\item This follows from the fact that
			\[
				(\d f - \iota_X\sigma)(u)\wedge\eta = \iota_u\big[\iota_X\omega - \d(f\eta)\big] = 0,	\hspace{.8cm}u\in\F_x.
			\]

		\item For $X\in T_xL$, then the previous assertion yields $\d j^*f = j^*\iota_X\sigma = \iota_X j^*\sigma$.
	\end{enumerate}
\end{proof}

\begin{remark}
	A similar argument shows that if $\F$ has constant rank, then $\iota\sigma:\F\to T^*M$ is an inclusion of vector bundles. The closed form $j^*\sigma$ is nondegenerate, and hence symplectic, when restriction along the fibers $T^*M\to \F^*$ yields an isomorphism $\F\to T^*M\to \F^*$.
\end{remark}


\section{Hamiltonian $G$-spaces}\label{sec:Hamiltonian_systems}

We begin with the fundamental construction of this paper.

\begin{definition}\label{def:moment_map}
	A \emph{comoment map} for a Lie algebra action $\g\curvearrowright M$ is a homomorphism of Leibniz algebra $\tilde\mu:\g\to\Omega_H^{k-1}(M)$ which completes the following commutative diagram.
	\begin{center}
	\begin{tikzpicture}
		\node (A) at (0,0) {$\g$};
		\node (B) at (2,0) {$\X(M)$};
		\node (C) at (2,2) {$\Omega_H^{k-1}(M)$};

		\node (D) at (0,-.6) {$\xi$};
		\node (E) at (2,-.6) {$\underline\xi$};

		\node (F) at (3.1,2) {$\alpha$};
		\node (G) at (3.1,0) {$X_\alpha$};

		\draw[->] (A) to (B);
		\draw[->,dashed] (A) to node[above left] {$\tilde\mu$} (C);
		\draw[->] (C) to (B);

		\draw[|->] (D) to (E);
		\draw[|->] (F) to (G);
	\end{tikzpicture}
	\end{center}
	The \emph{moment map} associated to $\tilde\mu$ is the differential form $\mu\in\Omega^{k-1}(M,\g^*)$ defined at each $x\in M$ by
	\begin{align*}
		\mu(x):	\g	&\to		\Lambda^{k-1}T_x^*M		\\
			\xi	&\mapsto	\tilde\mu(\xi)(x),
	\end{align*}
	under the natural identification of $\g^*\otimes\Lambda^{k-1}T_x^*M$ with $\Hom(\g,\Lambda^{k-1}T_x^*M)$. Together these data constitute a multisymplectic \emph{Hamiltonian $\g$-space} $(M,\omega,\g,\mu)$.

	If, additionally, $\g\curvearrowright M$ is induced by a multisymplectic Lie group action $G\curvearrowright M$, and if $\mu$ is $G$-equivariant as a map from $M$ to $\Lambda^{k-1}T^*M\otimes\g^*$, then we say that $\mu$ is a moment map for $G\curvearrowright M$ and that $(M,\omega,G,\mu)$ is a multisymplectic \emph{Hamiltonian $G$-space}.
\end{definition}

\begin{remark}
	\begin{enumerate}
		\item Here we consider that action of $G$ on $\g^*\otimes\Lambda^{k-1}T^*M$ given as the tensor product of the coadjoint action $G\curvearrowright\g^*$ and the induced action $G\curvearrowright\Lambda^{k-1}T^*M$.
		\item Except when explicitly stated otherwise, we will always take the moment map $\mu$ to be associated to a multisymplectic Lie group action $G\curvearrowright M$.
		\item There are numerous inequivalent generalizations of moment maps to the multisymplectic setting in the literature. Our definition of a Hamiltonian $G$-space $(M,\omega,G,\mu)$ generalizes the \emph{$\mathrm{Ad}^*$-equivariant covariant momentum maps} of \cite{GotayIsenbergMarsdenMontgomery1999} from the physical context (see also \cite[Equation 4.50]{MarsdenPatrickShkoller98}), and coincides with the construction of \cite{Echeverria-EnriquezMunoz-LecandaRoman-Roy18} consisting of a \emph{Coad-equivariant action} of $G$ on a multisymplectic manifold $(M,\omega)$ together with a \emph{momentum map} $\mu$.

		Other constructions include \emph{multi-moment maps} \cite{MadsenSwann12,MadsenSwann13}, \emph{homotopy moment maps} \cite{CalliesFregierRogersZambon16}, and \emph{weak moment maps} \cite{Herman18,Herman18a}. See \cite{MammadovaRyvkin20} for a comparison of the latter two.
	\end{enumerate}
\end{remark}

If $\mu\in\Omega^{k-1}(M,\g^*)$ is a moment map, the associated comoment map is given by $\tilde\mu(\xi) = \mu_\xi$ for each $\xi\in\g$. That is, $\tilde\mu(\xi)$ is the fiberwise contraction of $\mu$ and $\xi$.

\begin{example}\label{eg:standard_examples_Hamiltonian_systems}
	Consider again the multisymplectic manifolds of Example \ref{eg:standard_examples_manifolds}.
	\begin{enumerate}[i.]
		\item If $(M^{2n},\sigma,G,\nu)$ is a symplectic Hamiltonian $G$-space then $(M,\sigma^\ell,G,\,\ell\nu\wedge\sigma^{\ell-1})$ is a $(2\ell-1)$-plectic Hamiltonian $G$-space for $1\leq\ell\leq n$.
		\item Let $(G,\omega)$ be a semisimple Lie group with the canonical bi-invariant $2$-plectic structure. Since $\omega$ is bi-invariant, the left regular action of $G$ on itself is multisymplectic. The fundamental vector field $\underline\xi\in\X(G)$ associated to $\xi$ is the right invariant vector field extending $-\xi\in T_1G$. Using the $\Ad$-invariance of $\langle\,,\rangle$ we obtain
			\[
				\iota_\xi\omega = \langle \d\bar\theta,\bar\theta(\underline\xi)\rangle = \d\langle\bar\theta,\xi\rangle,
			\]
			from which $\langle\bar\theta,\xi\rangle\in\Omega_H^1(G)$ is a Hamiltonian $1$-form associated to $\underline\xi$. Since
			\[
				\L_\xi \,\langle\bar\theta,\zeta\rangle = -\langle[\xi,\bar\theta],\zeta\rangle = \langle\bar\theta,[\xi,\zeta]\rangle
			\]
			for every $\xi,\zeta\in\g$, it follows that $\langle\bar\theta,\cdot\,\rangle\in\Omega^1(G,\g^*)$ is a moment map for the left regular action of $G$. Likewise, $\langle\theta,\cdot\,\rangle$ is a moment map for the right regular action, and $\langle\bar\theta-\theta,\cdot\,\rangle$ is a moment map for the adjoint action.

		\item The smooth action of a Lie group $G$ on a fiber bundle $E\to\Sigma$ lifts to an action on $\pi:\Lambda_\ell^k\hspace{1pt}T^*E\to E$ in such a way that the canonical $k$-form $\theta\in\Omega^k(\Lambda_\ell^k\hspace{1pt}T^*E)$ is preserved. Explicitly, we define $(g\gamma)(X_1,\ldots,X_k) = \gamma(g^{-1}_*X_1,\ldots,g^{-1}_*X_k)$ for $\gamma\in\Lambda_\ell^k\hspace{1pt}T_x^*E$, $X_i\in T_{gx}E$, and $g\in G$. In particular, note that $\pi_*\underline\xi = \underline\xi_E$, where we write $\underline\xi$ for $\underline\xi_{\Lambda_\ell^k\hspace{1pt}T^*E}$. From
			\[
				\d\iota_\xi\theta = \L_\xi\theta - \iota_\xi\d\theta = -\iota_\xi\d\theta
			\]
			it follows that $\iota_\xi\theta\in\Omega^{k-1}(\Lambda_\ell^k\hspace{1pt}T^*E)$ is a Hamiltonian $(k-1)$-form associated to $\underline\xi\in\X(\Lambda_\ell^k\hspace{1pt}T^*E)$. Using again the $G$-invariance of $\theta$, we obtain
			\[
				\L_\xi \,\iota_\zeta\theta = \iota_{[\xi,\zeta]}\theta + \iota_\zeta\L_\xi\theta = \iota_{[\xi,\zeta]}\theta,
			\]
			and we conclude that
			\[
				\mu_\xi(\gamma) = (\iota_\xi\theta)_\gamma = \iota_{\xi_E}\gamma
			\]
			defines a moment map $\mu\in\Omega^{k-1}(\Lambda_\ell^k\hspace{1pt}T^*E,\g^*)$ for the action of $G$ on $\Lambda_\ell^k\hspace{1pt}T^*E$.

			More generally, any smooth action of a Lie group $G$ on an exact $k$-plectic manifold $(M,-\d\theta)$ that preserves the potential $\theta$ admits the moment map $\iota\hspace{1pt}\theta\in\Omega^{k-1}(M,\g^*)$.
	\end{enumerate}
\end{example}

\begin{example}
	A Hamiltonian form $\alpha\in\Omega_H^{k-1}(M)$ is said to be \emph{periodic} if the associated Hamiltonian vector field $X\in\X(M)$ generates an $S^1$-action on $(M,\omega)$. In this case, the assignment
	\begin{align*}
		\tilde\mu:	\R	&\to		\Omega^{k-1}(M)		\\
				t	&\mapsto	t\alpha
	\end{align*}
	is a comoment map precisely when $\L_X\alpha = \{\alpha,\alpha\} = [X,X] = 0$. Here we have identified $S^1$ with $\R/p\Z$ where $p>0$ is the period of $X$.
\end{example}

\begin{proposition}
	If $(M,\omega,\g,\mu)$ is a Hamiltonian $\g$-space, and if $G$ is connected, then $(M,\omega,G,\mu)$ is a Hamiltonian $G$-space.
\end{proposition}

\begin{proof}
	Fix $\xi\in\g$. From
	\[
		\L_\xi\mu_\zeta= \{\mu_\xi,\mu_\zeta\} = \mu_{[\xi,\zeta]}
	\]
	we deduce that $\langle\L_\xi\mu,\zeta\rangle = \langle\mathrm{ad}_{-\xi}^*\,\mu,\zeta\rangle$ for all $\zeta\in\g$ and the result follows as $G$ is connected.
\end{proof}

Indeed, for connected $G$ the identity $\d\mu_\xi=\iota_\xi\omega$ yields the equivalent characterization of the moment map as a $G$-equivariant form $\mu\in\Omega^{k-1}(M,\g^*)$ satisfying
\[
	\big\langle\d\mu(X_1,\ldots,X_k),\xi\big\rangle = (-1)^k\, \omega(X_1,\ldots,X_k,\underline\xi)
\]
for all $X_1,\ldots,X_k\in\X(M)$ and $\xi\in\g$. In this case, we will say that $\mu$ is \emph{equivariant}.

For the remainder of this section we restrict our attention to a particularly tractable class of moment maps, given as follows.

\begin{definition}
	The moment map $\mu\in\Omega^{k-1}(M,\g^*)$ is said to \emph{split} if $\mu=\nu\wedge\eta$ for some $\nu\in C^\infty(M,\g^*)$ and some closed $\eta\in\Omega^{k-1}(M)$. We say that $\mu=\nu\wedge\eta$ is an
	\begin{enumerate}[i.]
		\item \emph{invariant splitting} if $\eta$ is $G$-invariant,
		\item \emph{basic splitting} if $\eta$ is $G$-basic.
	\end{enumerate}
	We call the associated Hamiltonian $G$-space a \emph{split Hamiltonian $G$-space}.
\end{definition}

\begin{remark}
	\begin{enumerate}
		\item Here we are using the fact that the usual wedge product $\wedge:\Omega^*(M)\otimes\Omega^*(M)\to \Omega^*(M)$ extends naturally to a product
		\[
			\wedge : \Omega^*(M,\g^*)\otimes\Omega^*(M) \longrightarrow \Omega^*(M,\g^*).
		\]
		\item Recall that a form $\eta$ is \emph{$G$-basic} when it is both $G$-invariant and $G$-horizontal. In turn, the form $\eta$ is said to be \emph{$G$-horizontal} when $\iota_\xi\eta=0$ for all $\xi\in\g$.
	\end{enumerate}
\end{remark}

The closedness of $\eta$ implies that the kernel distribution $\F=\ker\eta\subset TM$ is involutive, since
\[
	\iota_{[X,Y]}\eta = [\L_X,\iota_Y]\hspace{1pt}\eta = -\iota_Y(\d\iota_X+\iota_X\d)\hspace{1pt}\eta = 0
\]
for all $X,Y\in\X(M)$ with $\iota_X\eta=\iota_Y\eta=0$. We make a simple but descriptive observation.

\begin{proposition}
	If the splitting $\mu=\nu\wedge\eta$ is
	\begin{enumerate}[i.]
		\item invariant, then $G$ preserves $\F$,
		\item basic, and if $G$ is connected, then $G$ fixes the leaves of $\F$.
	\end{enumerate}
\end{proposition}

\begin{proof}
	\begin{enumerate}[i.]
		\item This is immediate from the invariance of $\eta$.
		\item The condition $\iota_\xi\eta = 0$ is precisely that $\underline\xi_x\in\F$ for every $x\in M$.
	\end{enumerate}
\end{proof}

\begin{example}\label{eg:split_Hamiltonian_systems}
	\begin{enumerate}[i.]
		\item If $(M,\sigma^\ell,G,\ell\nu\wedge\sigma^{\ell-1})$ is obtained from $(M^{2n},\sigma,G,\nu)$ as in Example \ref{eg:standard_examples_Hamiltonian_systems}, above, then $\mu=\nu\wedge\ell\sigma^{\ell-1}$ is an invariant splitting, and is basic if and only if the action of $G$ is discrete. Each point of $M$ is a leaf of the kernel distribution of $\sigma^{\ell-1}$.
		\item Following \cite[Definition 1.2.6]{GuilleminLermanSternberg96}, we define a \emph{symplectic fibration} to be a smooth fiber bundle $\pi:M\to\Sigma$ equipped with a $2$-form $\sigma\in\Omega^2(M)$ which restricts to a symplectic structure on fibers and which satisfies
			\[
				\iota_Y\iota_X\d\sigma = 0
			\]
			for all vertical vectors $X,Y\in V_xM$ at $x\in M$. Suppose $\eta=\pi^*\d\mathrm{vol}_\Sigma\in\Omega^{k-1}(M)$ is the lift of a volume form on $\Sigma$. Since the span of any $k+2$ linearly independent tangent vectors at $x$ has at least $3$-dimensional intersection with $V_xM$, the condition above implies that $\d\sigma\wedge\eta=0$. Since $\sigma\wedge\eta$ is nondegenerate in both the vertical and in any complementary horizontal directions, we conclude that $\omega=\sigma\wedge\eta$ is a $k$-plectic structure on $M$.

			The kernel distribution of $\eta$ is the vertical tangent bundle $VM\subset TM$ of $\pi:M\to\Sigma$. The leaves of $\ker\eta$ coincide with the fibers of $\pi$.
	\end{enumerate}
\end{example}

\begin{proposition}
	If $\mu=\nu\wedge\eta$ is an invariant splitting, if $\omega=\sigma\wedge\eta$ for some $\sigma\in\Omega^2(M)$, and if $j:L\to M$ is a $G$-invariant integral manifold of the kernel distribution of $\eta$, then $(L,j^*\sigma,G,j^*\nu)$ is a presymplectic Hamiltonian $G$-space on $L$.
\end{proposition}

\begin{proof}
	This is an immediate consequence of Proposition \ref{prop:proportionally_closed_Hamiltonian_forms}.
\end{proof}

\begin{lemma}
	If $G$ is connected and $\mu=\nu\wedge\eta$ is an invariant splitting then the map
	\begin{align*}
		\tilde\nu:	\g	&\to		C^\infty(M)		\\
				\xi	&\mapsto	\:\nu_\xi
	\end{align*}
	is equivariant on the complement of the fixed point set of $G$.
\end{lemma}

\begin{proof}
	If $x\in M$ is not fixed by $G$, then there is a $\xi\in\g$ for which $\underline\xi_x\neq0$. The condition that
	\[
		\d\nu_\xi\wedge\eta = \iota_\xi\omega \neq 0
	\]
	implies that $\eta$ is nonzero at $x$. Since Definition \ref{def:moment_map} provides the equivariance of $\mu$, we conclude that
	\[
		\nu_{[\xi,\zeta]}\wedge\eta = \mu_{[\xi,\zeta]}=\L_\xi\mu_\zeta= \L_\xi\nu_\zeta\wedge\eta
	\]
	for all $\xi,\zeta\in\g$.
\end{proof}

\begin{proposition}
	If $\mu=\nu\wedge\eta$ is a splitting such that $\eta$ is $G$-horizontal, then 
	\begin{align*}
		\omega_\g:	\g	&\to		\Omega^*(M)	\\
				\xi	&\mapsto	\omega+\mu_\xi
	\end{align*}
	is an equivariantly closed differential form.
\end{proposition}

\begin{proof}
	From $\iota_\xi\mu=\nu\wedge\iota_\xi\eta=0$, we have
	\[
		\d_\g(\omega+\mu)(\xi) = \d\omega + (\d\mu_\xi - \iota_\xi\omega) - \iota_\xi\mu_\xi = 0
	\]
	for all $\xi\in\g$.
\end{proof}

We will have more to say about equivariant cohomology in Section \ref{sec:localization}. As in the symplectic setting \cite{Audin04}, we obtain stronger results for abelian group actions $T$.

\begin{proposition}
	If the action of a torus $T$ on a compact multisymplectic manifold $(M,\omega)$ admits a split moment map $\mu=\nu\wedge\eta$, then the fixed point set of $T$ is nonempty.
\end{proposition}

\begin{proof}
	Fix a generator $\xi\in\t$ of $T$, so that $\underline\xi$ generates the orbits of $T$ in $M$. Since $\mu_\xi=\nu_\xi\wedge\eta$ is proportional to the closed form $\eta$, Proposition \ref{prop:vanishing_vector_field} implies that $\underline\xi$ has nonempty vanishing set which, by the connectedness of $T$, constitutes the fixed point set of the action of $T$.
\end{proof}


\section{Reduction}\label{sec:reduction}

Let $(M,\omega,G,\mu)$ be a multisymplectic Hamiltonian $G$-space and let $\phi\in\Omega^{k-1}(M,\g^*)$ be a closed form. By identifying $\phi$ with its image in $\g^*\otimes\Lambda^{k-1}T^*M$, and by considering $\mu$ as a smooth function from $M$ to $\g^*\otimes\Lambda^{k-1}T^*M$, we denote by $\mu^{-1}(\phi)\subset M$ the set of points on which $\mu$ and $\phi$ agree. That is,
\[
	\mu^{-1}(\phi) = \{x\in M \,|\, \mu(x) = \phi(x)\}.
\]
In this setting, the isotropy subgroup $G_\phi\subset G$ consists of those elements $g\in G$ for which
\[
	\mathrm{Ad}_g^*\,\phi_{g^{-1}x}(g_*^{-1}X_1,\ldots,g_*^{-1}X_{k-1}) = \phi_x(X_1,\ldots,X_{k-1})
\]
for all $x\in M$ and $X_i\in T_xM$, where $\phi_{g^{-1}x}$ and $\phi_x$ denote the values of $\phi$ at $g^{-1}x$ and $x$, respectively.

We now present our main result.

\begin{theorem}[Multisymplectic reduction]\label{thm:reduction}
	Let $(M,\omega,G,\mu)$ be a $k$-plectic Hamiltonian $G$-space with moment map $\mu$, let $\phi\in\Omega^{k-1}(M,\g^*)$ be a closed form, and let $M_\phi= \mu^{-1}(\phi)/G_\phi$. If $\mu^{-1}(\phi)\subset M$ is an embedded submanifold and $G$ acts freely on $\mu^{-1}(\phi)$, then there is a unique, closed $\omega_\phi\in\Omega^{k+1}(M_\phi)$ satisfying $i^*\omega = \pi^*\omega_\phi$, where $i:\mu^{-1}(\phi)\to M$ is the inclusion and $\pi:\mu^{-1}(\phi)\to M_\phi$ is the quotient map.
	\begin{center}
		\begin{tikzpicture}
			\node (A) at (0,0) {$\mu^{-1}(\phi)$};
			\node (B) at (1.8,0) {$M$};
			\node (C) at (0,-1.8) {$M_\phi$};

			\draw[->] (A) to node[above] {$i$} (B);
			\draw[->] (A) to node[left] {$\pi$} (C);
		\end{tikzpicture}
	\end{center}
\end{theorem}

\begin{proof}
	First observe that the conditions on the inclusion $i$ and the action of $G$ imply that $M_\phi$ is a smooth manifold. The action of $G_\phi$ preserves $\mu^{-1}(\phi)$ by the equivariance of $\mu:M\to\Lambda^{k-1}T^*M\otimes\g^*$, and $i^*\phi=i^*\mu$ implies
	\[
		\iota_\xi \,i^*\omega = i^*\d\mu_\xi = i^*\d\phi_\xi= 0
	\]
	for all $\xi\in\g_\phi$. Thus we have shown $i^*\omega$ to be horizontal and invariant and it follows by the smoothness of the quotient map $\pi$ that $i^*\omega$ descends to $M_\phi$. Since $\pi_*:T\mu^{-1}(\phi)\to TM_\phi$ is surjective on fibers, it follows that the dual map $\pi^*:T^*M_\phi\to T\mu^{-1}(\phi)$ is injective. We conclude that $\omega_0$ is unique and that $\d\pi^*\omega_\phi=0$ implies $\d\omega_\phi=0$.
\end{proof}

We call $M_\phi$ the \emph{reduced space} and $\omega_\phi$ the \emph{reduced $(k+1)$-form} or the \emph{reduced premultisymplectic structure}. Note that $\omega_\phi$ is not necessarily a multisymplectic structure, since it may be degenerate. This property is shared with the theory of polysymplectic reduction \cite{MarreroRoman-RoySalgadoVilarino15}.

\begin{remark}
	If $\omega\in\Omega^2(M)$ is a symplectic structure and $\phi\in C^\infty(M,\g^*)$ has constant value $\lambda\in\g^*$. Then $G_\phi$ is the stabilizer of $\lambda$ under the coadjoint action on $\g^*$ and $M_\phi= \mu^{-1}(\lambda)/G_\lambda$ is the usual symplectic reduced space.
\end{remark}

\begin{remark}\label{rem:multisymplectic_reduction_remark}
	The proof of Theorem \ref{thm:reduction} does not invoke the nondegeneracy or homogeneity of the multisymplectic structure. Indeed, the result extends naturally to a tuple $(M,\omega,G,\mu)$ consisting of a smooth manifold $M$, a form $\omega\in\Omega^*(M)$, an action of a Lie group $G$ on $M$ preserving $\omega$, an equivariant form $\mu\in\Omega^*(M,\g^*)$ satisfying the identity $\d\mu_\xi=\iota_\xi\omega$, and a closed form $\phi\in\Omega^*(M,\g^*)$ satisfying the condition that $\mu^{-1}(\phi)\subset M$ is an embedded submanifold on which $G$ restricts to a free action. The result is a smooth \emph{reduced space} $M_\phi=\mu^{-1}(\phi)/G$ equipped with a \emph{reduced form} $\omega_\phi\in\Omega^*(M_\phi)$.
\end{remark}

\begin{example}\label{eg:standard_examples_reduction}
	We continue the analysis of Examples \ref{eg:standard_examples_manifolds} and \ref{eg:standard_examples_Hamiltonian_systems}.
	\begin{enumerate}[i.]
		\item The reduction of $(M,\sigma^\ell,G,\,\ell\nu\wedge\sigma^{\ell-1})$ at the level $\phi=\ell\lambda\wedge\sigma^{\ell-1}$, where we identify $\lambda\in\g^*$ with the corresponding constant function on $M$, is the symplectic reduced space $M_\lambda$ equipped with the premultisymplectic structure $\sigma_\lambda^\ell$, which is either nondegenerate or constantly zero.

		\item Let $T\subset G$ be a subtorus of a compact connected Lie group $G$. Thus, $\theta|_\t,\bar\theta|_\t\in\Omega^1(G,\t^*)$ are closed and
			\[
				\ker \hspace{1pt}(\bar\theta - \theta)|_\t = \ker\hspace{1pt}\mathrm{Ad}|_\t = N_GT,
			\]
			where $N_GT\subset G$ is the normalizer of $T$ in $G$, and where $\Ad|_\t:G\to\mathrm{End}(\t,\g)$ is the application of the adjoint map on $\t\subset\g$. If $T$ acts on $G$ by left multiplication, with moment map $\mu = \langle\bar\theta|_\t,\,\rangle\in\Omega^1(G,\t^*)$, then the $T$-invariance of $\theta|_\t$ and the property that
			\[
				\mu^{-1}(\theta|_\t) = \big\{\theta|_\t = \bar\theta|_\t\big\} = N_GT
			\]
			imply that the reduced space is $G_{\theta|_\t} = N_G(T)/T$. In particular, if $T$ is a maximal torus then the reduced space is given by the Weyl group $W(G,T)$ and the reduced premultisymplectic structure is the zero form.
		\item Consider the lift of an action of $G$ on $E\to\Sigma$ to the total space of $\Lambda_\ell^k\hspace{1pt}T^*E\to E$. The $0$-level set of the canonical moment map $\mu:\gamma\mapsto-\iota_{\xi_E}\gamma$ contains precisely the $G$-horizontal elements of $\Lambda_\ell^k\hspace{1pt}T^*E$, with respect to the action of $G$ on $E$, and thus the associated reduced space is $(\Lambda_\ell^k\hspace{1pt}T^*E)_0 = (\Lambda_\ell^k\hspace{1pt}T^*E)_{\text{$G$-hor.}}/G$.
	\end{enumerate}
\end{example}

\begin{theorem}[Reduction of Dynamics]\label{thm:reduction_of_dynamics}
	Let $(M,\omega,G,\mu)$ and $\phi$ be given as in Theorem \ref{thm:reduction} and let $\alpha\in\Omega_H^{k-\ell}(M)$ be $G$-invariant with associated Hamiltonian $\ell$-vector field $X\in\X^\ell(M)$. If
	\begin{enumerate}[i.]
		\item $X$ is $G$-invariant and tangent to $\mu^{-1}(\phi)$, or
		\item $\ell=1$ and $\{\alpha,\mu_\xi\}=\{\alpha,\phi_\xi\}=0$ for all $\xi\in\g$,
	\end{enumerate}
	then $X$ and $\d\alpha$ descend to $\bar{X}\in\X^\ell(M_\phi)$ and $\overline{\d\alpha}\in\Omega^{k-\ell}(M_\phi)$, respectively, and $\overline{\d\alpha} = \iota_{\bar X}\omega_\phi$.
\end{theorem}

Note that the brackets in case ii.\ are well-defined since in the $\ell=1$ setting $\alpha$ is a Hamiltonian $(k-1)$-form.

\begin{proof}
	\begin{enumerate}[i.]
		\item It suffices to consider $X=X_1\wedge\ldots\wedge X_\ell\in\X^\ell(M)$ where each $X_i\in\X(M)$ is $G$-invariant and tangent to $\mu^{-1}(\phi)$. Invariance implies that the restriction of $X$ to $\mu^{-1}(\phi)$ descends by $\pi:\mu^{-1}(\phi)\to M_\phi$ to $\pi_*X_1\wedge\ldots\wedge\pi_*X_\ell$. Using the identity $i^*\mu=i^*\phi$, we further deduce that
		\[
			\iota_\xi i^*\d\alpha= i^*\iota_\xi\iota_X\omega = -\iota_X i^*\d\phi_\xi = 0
		\]
		and it follows again by equivariance that $\d\alpha$ descends to $M_\phi$. The identity $\overline{\d\alpha} = \iota_{\bar X}\omega_\phi$ is clear.
	\item The equality $\L_X(\mu_\xi - \phi_\xi) = \{\alpha,\mu_\xi-\phi_\xi\} = 0$ implies that $X$ is tangent to $\mu^{-1}(\phi)$, and the result follows by part i.
	\end{enumerate}
\end{proof}

We call $\bar{X}$ the \emph{reduced multivector field}, or the \emph{reduced dynamics}, associated to $\alpha$. If $\alpha$ descends to $\bar\alpha\in\Omega^{k-\ell}(M_\phi)$ then we call $\bar\alpha$ the \emph{reduced Hamiltonian form}.

\begin{proposition}\label{prop:split_moment_map_free_action}
	If $\mu=\nu\wedge\eta$ is an invariant splitting, and if $\lambda\in\g^*$ is chosen so the action of $G$ on $\mu^{-1}(\lambda\wedge\eta)$ is locally free, then
	\begin{enumerate}[i.]
		\item $\nu^{-1}(\lambda) = \mu^{-1}(\lambda\wedge\eta)$,
		\item $\lambda$ is a regular value of $\nu$, in particular $\nu^{-1}(\lambda)\subset M$ is smooth, and
		\item $M_{\lambda\eta} = \nu^{-1}(\lambda)/G_\lambda$.
	\end{enumerate}
\end{proposition}

\begin{proof}
	\begin{enumerate}[i.]
		\item Clearly $\nu^{-1}(\lambda)\subset \mu^{-1}(\lambda\wedge\eta)$. The freeness of the induced Lie algebra action $\xi\mapsto\underline\xi$ along $\mu^{-1}(\lambda\wedge\eta)$ implies that
			\[
				\d\nu_\xi\wedge\eta = \iota_\xi\omega \neq 0,	\hspace{1.8cm}\xi\in\g\backslash\{0\}
			\]
			and consequently that $\eta$ is nonvanishing along $\mu^{-1}(\lambda\wedge\eta)$. The reverse inclusion follows.
		\item Part i.\ implies that the action of $G$ is locally free on $\nu^{-1}(\lambda)$. Thus, if $x\in\nu^{-1}(\lambda)$ then $\nu_*:T_xM\to T_\lambda\g^*$ is surjective since the condition
			\[
				\d\nu_\xi\wedge\eta = \iota_\xi\omega \neq 0,	\hspace{1.8cm}\xi\in\g\backslash\{0\}
			\]
			at $x$ implies that the dual map
			\begin{align*}
				\g	&\to		T_xM	\\
				\xi	&\mapsto	\langle\nu_*,\xi\rangle = \d\nu_\xi
			\end{align*}
			is injective.
		\item We deduce from the $G$-invariance of $\eta$ that $G_{\lambda\eta}=G_\lambda$, and the result follows by part i.
	\end{enumerate}
\end{proof}


\section{Variation of the reduced space}\label{sec:DH}

Let $(M,\omega,G,\mu)$ be a multisymplectic Hamiltonian $G$-space. Our aim in this section is to investigate the dependence of the reduced space $M_\phi=\mu^{-1}(\phi)/G_\phi$ on the closed form $\phi\in\Omega^{k-1}(M,\g^*)$. Our notation and approach follow \cite[Section 2]{DuistermaatHeckman82}.

We begin with an example.

\begin{example}
	Let $T$ be a torus and consider the $(2\ell-1)$-plectic Hamiltonian $T$-space $(M^{2n},\sigma^\ell,T,\ell\nu\wedge\sigma^{\ell-1})$ associated to the symplectic Hamiltonian $T$-space $(M,\sigma,T,\nu)$. The Duistermaat--Heckman theorem \cite[Theorem 1.1]{DuistermaatHeckman82} asserts that, for $\lambda$ and $\tau$ in the same connected component $C\subset\g^*$ of regular values of $\nu$,
	\[
		[\sigma_\lambda] = [\sigma_\tau] + \langle c,\lambda-\tau\rangle.
	\]
	Thus, according to Example \ref{eg:standard_examples_reduction}, the cohomology class of the reduced form
	\[
		[\sigma_{\ell\lambda\sigma^{\ell-1}}^\ell] = \big([\sigma_\tau] + \langle c,\lambda-\tau\rangle\big)^\ell
	\]
	exhibits polynomial dependence on $\lambda\in C$.
\end{example}

Before specializing to torus actions, let us first consider the general situation. Fix two $G$-invariant closed forms $\phi,\psi\in\Omega^{k-1}(M,\g^*)$ such that $\mu^{-1}(\phi)$ and $\mu^{-1}(\psi)\subset M$ are embedded submanifolds on which $G$ acts freely, let $\epsilon>0$, let $\ell\subset\Omega^{k-1}(M,\g^*)$ be the affine hull of $\phi-\epsilon\psi$ and $\phi+\epsilon\psi$ in $\Omega^{k-1}(M,\g^*)$, and suppose the diagram
\begin{center}
	\begin{tikzpicture}
		\node (A) at (0,0) {$\mu^{-1}(\ell)$};
		\node (B) at (2.5,0) {$\mu^{-1}(\phi) \times \ell$};
		\node (C) at (0,-1.8) {$\ell$};

		\draw[->] (A) to node[above] {$\sim$} (B);
		\draw[->] (A) to node[left] {$\mu$} (C);
		\draw[->] (B) to node[below right] {$\pi_2$} (C);
	\end{tikzpicture}
\end{center}
endows $\mu^{-1}(\ell)\to\ell$ with the structure of a trivialized fiber bundle with typical fiber $\mu^{-1}(\phi)$. As the diagram is equivariant we may take quotients by $G$ and descend to cohomology along the fibers to obtain
\begin{center}
	\begin{tikzpicture}
		\node (A) at (0,0) {$H_v^{k+1}(M_\ell)$};
		\node (B) at (2.5,0) {$H^{k+1}(M_\phi) \times \ell$};
		\node (C) at (0,-1.8) {$\ell$};

		\draw[->] (A) to node[above] {$\sim$} (B);
		\draw[->] (A) to node[left] {$\bar\mu$} (C);
		\draw[->] (B) to node[below right] {$\pi_2$} (C);
	\end{tikzpicture}
\end{center}
where $H_v^*(M_\ell)$ denotes the $\mu$-vertical cohomology of $M_\ell=\mu^{-1}(\ell)/G$. A key observation of \cite{DuistermaatHeckman82} is that that this trivialization is independent of the original identification $\mu^{-1}(\ell)\overset{\sim}{\longrightarrow}\mu^{-1}(\phi)\times\ell$. Thus, while the variation of the reduced form in terms of this trivialization,
\[
	\partial_\psi\hspace{1pt}\omega_\phi = \frac{\d}{\d t}\Big|_{t=0}\,\omega_{\phi+t\psi} \in \Omega^{k+1}(M_\phi),
\]
depends on the particular identification of $\mu^{-1}(\ell)$ with $\mu^{-1}(\phi)\times\ell$, the variation of the cohomology, $\partial_\psi\hspace{1pt}[\omega_\phi] \in H^{k+1}(M_\phi)$, depends only on $\phi$ and $\psi$.

We will also consider the case in which the $1$-dimensional parameter space $\ell$ is replaced by a neighborhood $P$ on an affine subspace of $\Omega^{k-1}(M,\g^*)$ modeled on $\g^*$.

\begin{lemma}\label{lem:general_variation}
	If $\tilde\psi\in\X(\mu^{-1}(\ell))$ is identified with $\partial_\psi\in\X(\ell)$ under the trivialization $\mu^{-1}(\ell)\overset{\sim}{\longrightarrow}\mu^{-1}(\phi)\times\ell$, then
	\[
		\pi^*\partial_\psi\omega_\phi = \d\hspace{1pt}i^*\iota_{\tilde\psi}\omega,
	\]
	in terms of this trivialization, where $\pi:\mu^{-1}(\phi)\to M_\phi$ is the quotient map on the model space and $i:\mu^{-1}(\phi)\to M$ is the inclusion.
\end{lemma}

\begin{proof}
	This follows as
	\begin{align*}
		&\pi^*\partial_\psi \omega_\phi= \partial_\psi\pi^*\omega_\phi= \partial_\psi i^*\omega  \\
		&\hspace{.5cm}= i^*\L_{\tilde\psi}\omega = i^*\d\iota_{\tilde\psi}\omega = \d\hspace{1pt} i^*\iota_{\tilde\psi}\omega,
	\end{align*}
	which is precisely \cite[Equation 2.3]{DuistermaatHeckman82} in the symplectic setting.
\end{proof}

When $(M,\omega)$ is symplectic and the Lie group $G$ is a torus $T$, the term $i^*\iota_{\tilde\psi}\omega$ encodes a connection $1$-form on the $T$-principal bundle $\mu^{-1}(\phi)\to M_\phi$, so that the variation $\partial_\psi\omega_\phi$ arises as a curvature $2$-form on $M_\phi$. The cohomology class of $\partial_\psi\omega_\phi$ is the Chern class $c\in H^2(M,\t)$ of $\mu^{-1}(\phi)\to M_\phi$ and does not depend on the choice of connection $1$-form. See the original derivation \cite[Section 2]{DuistermaatHeckman82} for more details and \cite[Chapter XII]{KobayashiNomizu69} for relevant background on Chern--Weil theory. 

With this in mind, our approach is to introduce suitable auxiliary data on $(M,\omega,T,\mu)$ which will enable us to relate the Chern class of $\mu^{-1}(\phi)\to M_\phi$ to the variations of $[\omega_\phi]$. This will take the form of a \emph{strongly conjugate distribution} $\underline\g^*\subset TM$ to the fundamental distribution $\underline\g$, defined as follows.

\begin{definition}
	Let $(E,\omega)$ be a $k$-plectic vector space. We will say that two subspaces $U$ and $V\subset E$ are \emph{conjugate subspaces} if the pairing
	\begin{align*}
			U\times V	&\rightarrow	\Lambda^{k-1}E^*		\\
			(X,Y)		&\mapsto	\iota_Y\iota_X\omega
	\end{align*}
	is nondegenerate and of rank $1$. In this case, we say that any nonzero element $\eta\in\Lambda^{k-1}E^*$ in the image of this map \emph{conjugates} $U$ and $V$. If, additionally, there is a $\sigma\in\Lambda^2E^*$ such that $\iota_Y\iota_X\omega = \sigma(X,Y)\,\eta$ for every $X\in U$ and $Y\in V$ then we say that $U$ and $V$ are \emph{strongly conjugate subspaces}.
\end{definition}

If $\eta$ conjugates $U$ and $V\subset E$, then 
\[
	\langle X,Y\rangle \,\eta = \iota_Y\iota_X\omega
\]
defines a nondegenerate bilinear pairing $\langle\,,\rangle: U\times V\to\R$. Note that this pairing depends by a factor of $\pm1$ on the order in which $X$ and $Y$ are applied to $\omega$. If $U$ and $V$ are strongly conjugate, with $\iota_Y\iota_X\omega = \sigma(X,Y)\,\eta$, then $\sigma$ extends the pairing $\langle\,,\rangle$ to $E$ in a one-sided manner in the sense that the following diagrams commute.
\begin{center}
	\begin{tikzpicture}
		\node (A) at (0,0) {$U$};
		\node (B) at (2,0) {$V^*$};
		\node (C) at (2,1.5) {$E^*$};

		\node (D) at (0,-.6) {$X$};
		\node (E) at (2,-.6) {$\langle X,\:\rangle$};

		\draw[->] (A) to (B);
		\draw[->] (A) to node[above left] {$\iota\hspace{.5pt}\sigma$} (C);
		\draw[->] (C) to (B);

		\begin{scope}[shift={(5,0)}]
			\node (A) at (2,0) {$V$};
			\node (B) at (0,0) {$U^*$};
			\node (C) at (0,1.5) {$E^*$};

			\node (D) at (2,-.6) {$Y$};
			\node (E) at (0,-.6) {$\langle\;,Y\rangle$};

			\draw[->] (A) to (B);
			\draw[->] (A) to node[above right] {$\iota\hspace{.5pt}\sigma$} (C);
			\draw[->] (C) to (B);
		\end{scope}
	\end{tikzpicture}
\end{center}
Here the diagonal maps $Z\mapsto\iota_Z\sigma$ represent contraction with $\sigma$, and the vertical maps are the natural restrictions. Thus, for example, the resulting pairing of two elements $X,X'\in U$ is given by $\sigma(X,X')$.

\begin{definition}
	Two distributions $U$ and $V\subset TM$ are said to be \emph{conjugate distributions} if there is a closed form $\eta\in\Omega^{k-1}(M)$ which conjugates the fibers of $U$ and $V$ at every point of $M$. They are \emph{strongly conjugate distributions} if there is a $2$-form $\sigma\in\Omega^2(M)$ such that $\iota_Y\iota_X\omega=\sigma(X,Y)\,\eta$ for $X\in U_x$ and $Y\in V_x$ at every $x\in M$.
\end{definition}

If $G$ acts locally freely, so that the fibers of the fundamental distribution $\underline\g$ are linearly isomorphic to $\g$, and if $\eta\in\Omega^{k-1}(M)$ conjugates $\underline\g$ with a distribution $U\subset TM$ then the fibers of $U$ are naturally identified with $\g^*$. In this case, we write $\underline\lambda\in\X(M)$ for the image of $\lambda\in\g^*$ under this identification and we denote $U$ by $\underline\g^*$. Note that the assignment $\lambda\mapsto\underline\lambda$ depends on the choice of both $U$ and $\eta$.

We now specialize to abelian actions. Let $T$ be a torus.

\begin{lemma}\label{lem:multisymplectic_variation_main_lemma}
	Suppose a $T$-equivariant distribution $\underline\t^*\subset TM$ is strongly conjugate to $\underline\t$ on a neighborhood of $\mu^{-1}(\phi)\subset M$ with respect to a $T$-invariant form $\eta\in\Omega^{k-1}(M)$, and suppose that $T$ acts freely on $\mu^{-1}(\phi)$.
	\begin{enumerate}[i.]
		\item There is a $T$-invariant $\alpha\in\Omega^1(M,\t)$ with
			\[
				\iota_\lambda\omega = \alpha_\lambda \wedge\eta,	\hspace{1.5cm}\lambda\in\t^*
			\]
		\item $i^*\alpha\in\Omega^1(\mu^{-1}(\phi),\t)$ is a connection $1$-form on the $T$-principal bundle $\mu^{-1}(\phi)\to M_\phi$,
		\item $i^*\eta = \pi^*\eta_\phi$ for a unique, closed $\eta_\phi\in\Omega^{k-1}(M_\phi)$, and
		\item $\partial_\lambda[\omega_\phi] = \langle c,\lambda\rangle \wedge [\eta_\phi]$, where $c$ is the Chern class of $\mu^{-1}(\phi)\to M_\phi$, and $\partial_\lambda\in\X(\Omega^{k-1}(M,\t^*))$ is tangent to the affine action of $\t^*$ on $\Omega^{k-1}(M)$ given by $\lambda+\psi = \lambda\wedge\eta+\psi$.
	\end{enumerate}
\end{lemma}

\begin{proof}
	\begin{enumerate}[i.]
		\item The strong conjugacy condition guarantees a $\sigma\in\Omega^2(M)$ with $\iota_\xi\iota_\lambda\omega = \iota_\xi\iota_\lambda\sigma\wedge\eta$. Since $\t^*$ is equivariant, since $\eta$ and $\omega$ are invariant, and since $T$ is compact, we may assume that $\sigma$ is invariant by averaging over the action of $T$. The form $\alpha\in\Omega^1(M,\t)$ is defined at each $x\in M$ by
			\begin{align*}
				\alpha:	\t^*	&\to		T_x^*M			\\
					\lambda	&\mapsto	\iota_\lambda\sigma_x
			\end{align*}
			under the natural identification of $\t\otimes T_x^*M$ with $\Hom(\t^*,T_x^*M)$.
		\item This follows by the $T$-invariance of $\alpha$, the identity $\underline\xi_x  = -\frac{\d}{\d t}\hspace{1pt}e^{t\xi}\hspace{1pt} x \hspace{1pt}\big|_{t=0}$, and the fact that
			\[
				\langle\alpha(\underline\xi),\lambda\rangle \,\eta= \alpha_\lambda(\underline\xi)\wedge\eta = \iota_\xi\iota_\lambda\omega = -\langle\xi,\lambda\rangle\,\eta
			\]
			implies $i^*\alpha(\underline\xi)=\alpha(\underline\xi)=-\xi$. Note that the freeness of $T$ ensures that $\eta$ is nowhere vanishing on $\mu^{-1}(\phi)$.
		\item Since $\eta_x=\iota_\xi\iota_\lambda\omega_x$ for some $\xi\in\g$ and $\lambda\in\g^*$, we have
			\[
				\iota_\xi i^*\eta_x = i^*\iota_\xi\iota_\xi\iota_\lambda\omega_x = 0,
			\]
			so that $i^*\eta$ descends along $\pi:\mu^{-1}(\phi)\to M_\phi$ by equivariance. The uniqueness and closedness of $\eta_\phi$ both follow from the injectivity of $\pi^*:TM_\phi\to T\mu^{-1}(\phi)$ and the closedness of $i^*\eta$.
		\item Using part ii., part iii., and the closedness of $\eta$, we obtain
			\[
				\d i^*\iota_\lambda\omega = \d i^*(\alpha_\lambda\wedge\eta)
					= \langle\d i^*\alpha,\lambda\rangle \wedge i^*\eta = \pi^*\langle F_\alpha,\lambda\rangle\wedge\pi^*\eta_\phi.
			\]
			By Lemma \ref{lem:general_variation} and the injectivity of $\pi^*$, we conclude that
			\[
				\partial_\lambda[\omega_\phi] = \langle[F_\alpha],\lambda\rangle \wedge [\eta_\phi] = \langle c,\lambda\rangle \wedge [\eta_\phi],
			\]
		where $F_\alpha=\d i^*\alpha\in\Omega^2(M_\phi,\t)$ is the curvature of the connection $1$-form $i^*\alpha$.
	\end{enumerate}
\end{proof}

Consolidating the above developments, we arrive at the main result of this section.

\begin{theorem}[Variation of the multisymplectic reduced space]\label{thm:multisymplectic_variation}
	Let $T$ be a torus, let $(M,\omega,T,\mu)$ be a $k$-plectic Hamiltonian $T$-space, fix a $T$-invariant closed form $\phi\in\Omega^{k-1}(M,\t^*)$ such that $\mu^{-1}(\phi)\subset M$ is an embedded submanifold on which $T$ acts freely, choose an open subset $C\subset\t^*$, let $\eta\in\Omega^{k-1}(M)$ be $T$-invariant, and write $P=C\wedge\eta+\phi$. If
	\begin{enumerate}[i.]
		\item the diagram
			\begin{center}
			\begin{tikzpicture}
				\node (A) at (0,0) {$\mu^{-1}(P)$};
				\node (B) at (2.5,0) {$\mu^{-1}(\phi) \times P$};
				\node (C) at (0,-1.8) {$P$};

				\draw[->] (A) to node[above] {$\sim$} (B);
				\draw[->] (A) to node[left] {$\mu$} (C);
				\draw[->] (B) to node[below right] {$\pi_2$} (C);
			\end{tikzpicture}
			\end{center}
			is a trivialization of a family of $T$-principal bundles modeled on $\mu^{-1}(\phi)$, and
		\item the fundamental distribution $\underline\t$ is strongly conjugate to a distribution $\underline\t^*\subset TM$ with respect $\eta$,
	\end{enumerate}
	then, 
	\[
		\partial_\lambda\, [\omega_\psi] = \langle c,\lambda\rangle \wedge [\eta_\psi],		\hspace{1.5cm}\lambda\in C,\;\psi\in P
	\]
	where $c\in\Omega^2(M_\phi,\t)$ is the Chern class of the model space $\mu^{-1}(\phi)\to M_\phi$.
\end{theorem}

\begin{proof}
	This follows from of Lemma \ref{lem:multisymplectic_variation_main_lemma} part iv.\ and the observation that the Chern form $c\in H^2(M_\phi,\t)$, as an invariant of $T$-principal bundles, does not depend on the choice of model space $\mu^{-1}(\phi)\to M_\phi$ for $\phi\in P$.
\end{proof}


\section{Localization for split Hamiltonian $G$-spaces}\label{sec:localization}

In this section we collect some observations and results relating to the interaction of equivariant localization with multisymplectic geometry. The statements and proofs are straightforward adaptations of corresponding results in symplectic geometry, and can serve as guides for further generalizations. We refer to \cite[Chapter 7]{BerlineGetzlerVergne92} and \cite[Chapter 9]{DwivediHermanJeffreyvandenHurk19} for background on equivariant differential forms.

\begin{lemma}\label{lem:basic_splitting_equivariant_measure}
	If $\mu=\nu\wedge\eta$ is a basic splitting with respect to a Hamiltonian action $G\curvearrowright(M,\omega)$, and if $\omega=\sigma\wedge\eta$ for some $\sigma\in\Omega^2(M)$, then $e^{z(\sigma+\nu)}\eta\in\Omega_\g^*(M)$ is equivariantly closed for all $z\in\C$.
\end{lemma}

\begin{proof}
	From $\d_\g\eta=0$ and
	\[
		\d_\g(\sigma+\nu)(\xi)\hspace{1pt}\wedge\eta = (\d\sigma+\d\nu_\xi - \iota_\xi\sigma)\wedge\eta = 0,\hspace{1.1cm}\xi\in\g
	\]
	we obtain
	\[
		\d_\g \big[z^\ell(\sigma+\nu)^\ell\wedge\eta\big] = \ell\hspace{1pt}z^\ell\hspace{1pt}(\sigma+\nu)^{\ell-1}\wedge\hspace{1pt}\d_\g(\sigma+\nu)\hspace{1pt}\wedge\eta =0,	\hspace{1cm}\ell\geq1.
	\]
	The result follows by summing over $\ell\geq0$, since $\d_\g\big[(\sigma+\nu)^0\wedge\eta\big]=0$.
\end{proof}

The following corollary is a adaptation of \cite[Theorem 10.11]{DwivediHermanJeffreyvandenHurk19}. See \cite[Section 7.4]{BerlineGetzlerVergne92} for more general results in the symplectic setting, and \cite[Section 33]{GuilleminSternberg84} for background on the physical context for stationary phase approximations more generally.

\begin{theorem}[Exact stationary phase approximation]
	If $(M,\omega,T,\mu)$ is a $k$-plectic Hamiltonian $T$-space with $M$ compact and $T$ a torus, if $\mu=\nu\wedge\eta$ is a basic splitting with nowhere vanishing $\eta\in\Omega^{k-1}(M)$, if each component $F\subset M$ of the fixed point set of $T$ is tangent to $\ker\eta\subset TM$ in the sense that $\iota_X\eta=0$ for all $X\in TF$, and if $\omega=\sigma\wedge\eta$ for some $\sigma\in\Omega^2(M)$, then
	\[
		\int_M e^{\i\nu_\xi}\,e^{\i\sigma}\eta = \sum_{F\in\mathcal{F}} e^{\i\nu_\xi(F)} \int_F\frac{e^{\i\sigma}}{e_F(\xi)}\,\eta
	\]
	for all generators $\xi\in\t$ of $T$, where $\mathcal{F}$ contains the connected components of the fixed point set of $T$, and where $e_F$ is the equivariant Euler class of $F$.
\end{theorem}

Here we say that $\xi\in\t$ is a \emph{generator} of $T$ when $\mathrm{exp}(\R\xi)$ is dense in $T$.

\begin{proof}
	Fix $F\in\mathcal{F}$. Since $\underline\xi=0$ on $F$, and since $F$ is tangent to $\ker\eta$, it follows that for all $X\in TF$ we have
	\[
		(X\nu_\xi)\wedge\eta = \iota_X(\d\nu_\xi\wedge\eta) = \iota_X\iota_\xi\omega = 0,
	\]
	so that $X\nu_\xi=0$. Consequently, $\nu_\xi(F)\in\R$ is well defined. Since $\xi$ generates $T$, the vanishing set of $\underline\xi\in\X(M)$ is precisely the fixed point set of $T$, and the result follows by Lemma \ref{lem:basic_splitting_equivariant_measure} and the equivariant localization theorem \cite[Theorem 9.50]{DwivediHermanJeffreyvandenHurk19}.
\end{proof}

Suppose that $\mu=\nu\wedge\eta$ splits, that $\nu^{-1}(0)\subset M$ is an embedded submanifold on which $G$ acts freely, and that $\eta$ conjugates $\underline\g$ with some distribution $\underline\g^*\subset TM|_{\nu^{-1}(0)}$ along $\nu^{-1}(0)$. If additionally $\omega=\sigma\wedge\eta$ for some $\sigma\in\Omega^2(M)$, if the form $\alpha\in\Omega^1(M,\g)|_{\nu^{-1}(0)}$ given at each $x\in\nu^{-1}(0)$ by
	\begin{align*}
		\alpha:	\g^*	&\to		T^*M			\\
			\lambda	&\mapsto	\iota_\lambda\sigma,
	\end{align*}
	under the natural identification of $\g\otimes T_x^*M$ with $\Hom(\g^*,T_x^*M)$, is $G$-equivariant with respect to the adjoint action of $G$ on $\g$, and if $\alpha$ vanishes on $\underline\g^*$, then we will say that $\alpha$ is the (extended) \emph{connection $1$-form associated to $\sigma$}. This terminology is chosen in light of the equivariance of $\alpha$ and the fact that the identity $\alpha_\lambda=\iota_\lambda\sigma$ yields
	\[
		\langle\alpha(\underline\xi),\lambda\rangle = \sigma(\underline\lambda,\underline\xi) = -\langle \xi,\lambda\rangle,
	\]
	from which $\alpha(\underline\xi)=-\xi$ for all $\xi\in\g$. We equip $\g$ with a $G$-invariant metric, which exists by the compactness of $G$, and we endow $\g^*$ with its dual metric.

\begin{lemma}\label{lem:split_moment_map_structure}
	We have
	\begin{enumerate}[i.]
		\item $\iota_\g\eta = \iota_{\g^*}\eta=0$.
		\item $\nu_*\underline\lambda=\lambda$ for all $\lambda\in\g^*$,
		\item $\underline\g^*$ is a normal bundle to $\nu^{-1}(0)\subset M$, and
	\end{enumerate}
	Moreover, if $\omega=\sigma\wedge\eta$ for some $\sigma\in\Omega^2(M)$ with associated connection $1$-form $\alpha\in\Omega^1(M,\g)|_{\nu^{-1}(0)}$, and if $\sigma\in\Omega^2(M)$ descends to $\sigma_0\in\Omega^2(M_0)$, then
	\begin{enumerate}[i.]\setcounter{enumi}{3}
		\item $\d\langle\alpha,\nu\rangle(\underline\xi,\underline\lambda) = \langle\xi,\lambda\rangle$,
		\item $\sigma=\pi^*\sigma_0$ on $A=\ker\alpha$, and
		\item $\sigma=\d\langle\alpha,\nu\rangle$ on $\underline\g\oplus\underline\g^*$.
	\end{enumerate}
\end{lemma}

\begin{proof}
	First recall that Proposition \ref{prop:split_moment_map_free_action} ensures $\nu^{-1}(0)=\mu^{-1}(0)$, and that $\eta$ is nowhere vanishing on $\nu^{-1}(0)$ since the action of $G$ on $\nu^{-1}(0)$ is free.
	\begin{enumerate}[i.]
		\item This is a consequence of $\iota_\lambda\iota_\xi\omega = \langle\xi,\lambda\rangle\,\eta$.
		\item From part i.\ and the identity $\iota_\xi\omega=\d\nu_\xi\wedge\eta$ we obtain
			\[
				\langle\xi,\lambda\rangle\,\eta = \iota_\lambda\iota_\xi\omega = \langle\xi,\nu_*\underline\lambda\rangle\,\eta.
			\]
		\item Proposition \ref{prop:split_moment_map_free_action} ensures that $\nu^{-1}(0)$ is smooth and the result follows by part ii.
		\item From the $G$-equivariance of $\alpha$ and the equality $\alpha(\underline\lambda)=0$, we obtain $(\L_\xi\alpha)(\underline\lambda) = -\mathrm{ad}_\xi\hspace{1pt}\alpha(\underline\lambda) = 0$. Consequently,
			\[
				\iota_\lambda\L_\xi\langle\alpha,\nu\rangle
					= \langle\L_\xi\alpha(\underline\lambda),\nu\rangle
						+ \langle\alpha(\underline\lambda),\L_\xi\nu\rangle
					= 0,
			\]
			and an application of part ii.\ yields
			\[
				\iota_\lambda\iota_\xi\d\langle\alpha,\nu\rangle
					= \iota_\lambda\L_\xi\langle\alpha,\nu\rangle + \iota_\lambda\d\langle\xi,\nu\rangle
					= \langle\xi,\nu_*\underline\lambda\rangle
					= \langle\xi,\lambda\rangle.
			\]
		\item The restriction $\pi_*|_A:A\to TM_0$ is a linear isomorphism on fibers, $\iota_\g i^*\sigma=0$, and $\pi^*\sigma_0=i^*\sigma$.
		\item Since $\underline\g$ and $\underline\g^*$ are strongly conjugate with respect to $\sigma$ and $\eta$, we have $\sigma(\underline\xi,\underline\lambda)=\langle\xi,\lambda\rangle$ and the result follows by part iv.
	\end{enumerate}
\end{proof}

We apply this lemma to obtain the following structure result.

\begin{proposition}
	If $\omega=\sigma\wedge\eta$ for some $\sigma\in\Omega^2(M)$ with associated connection $1$-form $\alpha\in\Omega^1(M,\g)|_{\nu^{-1}(0)}$, and if $\sigma$ descends to $\sigma_0\in\Omega^2(M_0)$, then
	\[
		TM|_{\nu^{-1}(0)}=A\oplus\underline\g\oplus\underline\g^*
	\]
	where $A=\ker\alpha$. If additionally $\omega=\sigma\wedge\eta$ then, in terms of the splitting $A\oplus(\underline\g\oplus\underline\g^*)$, we have
	\[
		\sigma|_{\nu^{-1}(0)} =
		\begin{pmatrix}
			\pi^*\sigma_0	&\gamma		\\
			-\gamma		&\d\langle\alpha,\nu\rangle|_{\underline\g\oplus\underline\g^*}
		\end{pmatrix}
	\]
	for some $\gamma\in\Gamma\big(\nu^{-1}(0),A^*\otimes(\underline\g\oplus\underline\g^*)^*\big)$.
\end{proposition}

\begin{proof}
	The first assertion follows from part iii.\ of Lemma \ref{lem:split_moment_map_structure}, the second from parts v.\ and vi.
\end{proof}

We now extend a nonabelian localization theorem of Liu, following very closely the derivations of \cite[Lemma 2]{Liu95} and \cite[Proposition 1]{Liu99}. This result was motivated originally in the symplectic setting by previous work of Witten \cite{Witten91} and Jeffrey--Kirwan \cite{JeffreyKirwan95}.

\begin{theorem}[Nonabelian localization]\label{thm:localization}
	Suppose that $(M,\omega,G,\mu)$ is a multisymplectic Hamiltonian $G$-space with $M$ and $G$ compact, that $\mu=\nu\wedge\eta$ is an invariant splitting, that $G$ is connected and acts freely on $\mu^{-1}(0)$, that $\omega=\sigma\wedge\eta$ for some $\sigma\in\Omega^2(M)$ with associated connection $1$-form $\alpha\in\Omega^1(M,\g)|_{\nu^{-1}(0)}$, and that $\sigma$ descends to $\sigma_0\in\Omega^2(M_0)$. Choose $\delta>0$ so that $\nu^{-1}(B_\delta)\cong\nu^{-1}(0)\times B_\delta$ as smooth $G$-manifolds, and let $\lambda:B_\delta\to\g^*$ be the inclusion.

	If $\sigma$ is identified with $\pi^*\sigma_0+\d\langle\alpha,\lambda\rangle$ on $\nu^{-1}(B_\delta)\cong\mu^{-1}(0)\times B_\delta$, then
	\[
		\int_\g e^{-t\|\xi\|^2} \int_M e^{\sigma+\i\nu_\xi}\hspace{1pt}\eta\,\d\xi
			\;=\; (2\pi)^\ell\;|G|\int_{M_0} e^{\sigma_0+t\|F\|^2}\eta_0 \,+\, O(e^{-\delta^2/4t}),
	\]
	where $\ell=\dim G$ and $F\in\Omega^2(M_0,\mathrm{ad}\,\nu^{-1}(0))$ is the curvature of $\alpha$.
\end{theorem}

\begin{remark}
	The Fourier transform of the measure associated to $\nu_*e^\sigma\eta$ on $\g^*$ is given by
	\[
		\widehat{\nu_*e^\sigma\eta}\,(\xi) = \int_M e^{\sigma-\i\nu_\xi}\hspace{1pt}\eta
	\]
	for $\xi\in\g$. By taking complex conjugates, we see that Theorem \ref{thm:localization} computes the integral of $\widehat{\nu_*e^\sigma\eta}$ over $\g$ with respect to the Gaussian measure $e^{-t\|\xi\|^2}\d\xi$.
\end{remark}

We devote the remainder of this section to proving Theorem \ref{thm:localization}. For all $t>0$ and $\lambda,\tau\in\g^*$, define the heat kernel
\[
	H(t,\lambda,\tau) = \frac{1}{(4\pi t)^{\ell/2}}\,e^{-\|\lambda-\tau\|^2/4t}
\]
and the integral
\[
	I(t) = \int_M H(t,\nu,0) \,e^\sigma\eta.
\]
Following \cite{Liu99}, our approach will be to compare each side of the equality in Theorem \ref{thm:localization} with $I(t)$. We begin with the left-hand side.

\begin{lemma}\label{lem:localization_lhs}
	For all $t>0$,
	\[
		I(t) = (2\pi)^{-\ell} \int_\g e^{-t\|\xi\|^2} \int_M e^{\sigma+\i\nu_\xi}\hspace{1pt}\eta\:\d\xi.
	\]
\end{lemma}

\begin{proof}
	An application of Lemma \ref{lem:gaussian_integral} part i., to follow, yields
	\[
		(4\pi t)^{-\ell/2} \int_M e^{-\|\nu\|^2/4t} e^\sigma\eta
			= (2\pi)^{-\ell} \int_M \int_\g e^{-t\|\xi\|^2+\i\nu_\xi}\,\d\xi \,e^\sigma\eta.
	\]
\end{proof}

\begin{lemma}\label{lem:gaussian_integral}
	If $V$ is an $\ell$-dimensional Euclidean vector space, then for all $y\in V$ and $t>0$,
	\begin{enumerate}[i.]
		\item 
		\[
			e^{-\|y\|^2/4t} = \left(\frac{t}{\pi}\right)^{\ell/2} \int_V e^{-t\|x\|^2+\i\langle x,y\rangle}\,\d x,
		\]
		\item 
		\[
			e^{t\|y\|^2} = (4\pi t)^{-\ell/2} \int_V e^{-\|x\|^2/4t} e^{\langle x,y\rangle}\,\d x.
		\]
	\end{enumerate}
\end{lemma}

See \cite[Equation 9.A.6]{Weinberg95} and the surrounding discussion for more general results.

\begin{proof}
	Let $\langle\,,\rangle$ denote the $\C$-bilinear extension of the Euclidean structure on $V$ to the complexification $V^\C$ and write $\|z\|^2=\langle z,z\rangle$ for $z\in V^\C$. Notwithstanding our notation, note that $\langle\,,\rangle$ is not a Hermitian structure on $V^\C$.
	\begin{enumerate}[i.]
		\item From
		\[
			t\|x\|^2 - \i\langle x,y\rangle = t\hspace{.5pt}\Big\|x-\frac{\i}{2t}y\Big\|^2 + \frac{1}{4t}\|y\|^2,	\hspace{1cm}x,y\in V
		\]
		we obtain
		\begin{align*}
			\int_V e^{-t\|x\|^2+\i\langle x,y\rangle}\,\d x
				&=	\int_V e^{-t\|x-\i y/2t\|^2} \, e^{-\|y\|^2/4t} \,\d x								\\
				&=	e^{-\|y\|^2/4t} \,t^{-\ell/2} \int_V e^{-\|x'\|^2}\,\d x', \hspace{1cm}x'=t^{1/2}\hspace{1pt}(x-\i y/2t)	\\
				&=	e^{-\|y\|^2/4t} \,(\pi/t)^{\ell/2}.
		\end{align*}
		\item The equality
		\[
			\|x\|^2 - 4t\hspace{1pt}\langle x,y\rangle = \|x-2ty\|^2 - \|2ty\|^2,	\hspace{1cm}x,y\in V
		\]
		yields
		\begin{align*}
			\int_V e^{-\|x\|^2/4t} e^{\langle x,y\rangle}\,\d x
				&= \int_V e^{-\|x-2ty\|^2/4t} \,e^{t\|y\|^2}\,\d x	 			\\
				&= e^{t\|y\|^2} (4t)^{\ell/2}\int_V e^{-\|x'\|^2}\,\d x', \hspace{1cm} x'=(x-2ty)/2t^{1/2}	\\
				&= e^{t\|y\|^2} (4\pi t)^{\ell/2}.
		\end{align*}
	\end{enumerate}
\end{proof}

We now consider the right-hand side. Note that Lemma \ref{lem:split_moment_map_structure} part i.\ implies that $\eta$ is basic.

\begin{lemma}\label{lem:localization_rhs}
	If $\pi^*\sigma_0+\d\langle\alpha,\lambda\rangle$ is identified with $\sigma$ on $\nu^{-1}(B_\delta)\cong\nu^{-1}(0)\times B_\delta$, then
	\[
		I(t) = |G|\int_{M_0} e^{\sigma_0+t\|F\|^2}\eta_0 \;+\; O(e^{-\delta^2/4t})
	\]
	as $t\to0$.
\end{lemma}

\begin{proof}
	We adapt the argument of \cite[Section 3]{Liu95}. Identifying the fibers of $\nu^{-1}(0)\to M_0$ with $G$ and choosing suitable orientations on $G$ and $\g^*$ yields
	\[
		e^{\langle\alpha,\d\lambda\rangle} = -\d\mathrm{vol}_G\,\d\mathrm{vol}_{\g^*} \;+\:\text{ lower degree terms}.
	\]
	Since $\langle[\alpha\wedge\alpha],\lambda\rangle^j \:e^{\langle \pi^*F,\lambda\rangle} e^{\langle\alpha,\d\lambda\rangle}$ vanishes at top degree for $j>0$, we have
	\[
		e^{\d\langle\alpha,\lambda\rangle}
			=	e^{\langle \pi^*F,\lambda\rangle} \,e^{-\langle[\alpha\wedge\alpha],\lambda\rangle} e^{-\langle\alpha,\d\lambda\rangle}
			=	e^{\langle \pi^*F,\lambda\rangle} \;\d\mathrm{vol}_G\;\d\mathrm{vol}_{\g^*} \;+\; \text{LDT}
	\]
	on $\nu^{-1}(0)\times B_\delta$. Applying Lemma \ref{lem:gaussian_integral} part ii.\ and Varadhan's formula for the short time asymptotics of the heat kernel \cite{Saloff-Coste10}, we obtain
	\begin{align*}
		\int_{\nu^{-1}(B_\delta)} H(t,\nu,0)\,e^{\sigma}\eta
		&=	(4\pi t)^{-\ell/2}\int_{\nu^{-1}(0)} e^{\pi^*\sigma_0}\eta\,\d\mathrm{vol}_G
			\int_{B_\delta} e^{-\|\lambda\|^2/4t}\,e^{\langle\pi^*F,\lambda\rangle}\,\d\mathrm{vol}_{\g^*}	\\
		&=	|G|\int_{M_0} e^{\sigma_0}\eta_0 \,
			\Big[(4\pi t)^{-\ell/2}\!\int_{B_\delta} e^{-\|\lambda\|^2/4t} e^{\langle F,\lambda\rangle}\,\d\mathrm{vol}_{\g^*}\Big]	\\
		&=	|G|\int_{M_0} e^{\sigma_0+t\|F\|^2}\eta_0 \;+\;O(e^{-\delta^2/4t}),
	\end{align*}
	and the result follows as
	\[
		\int_{M\backslash\nu^{-1}(B_\delta)} H(t,\nu,0)\,e^\sigma\eta = O(e^{-\delta^2/4t}).
	\]
\end{proof}

Lemmas \ref{lem:localization_lhs} and \ref{lem:localization_rhs} together yield Theorem \ref{thm:localization}.


\section{Outlook}\label{sec:outlook}

We suggest five directions for further development.

\begin{enumerate}
	\item \textbf{Nondegenerate and generalized reduction.} Perhaps the most immediate question is whether it is possible to obtain precise conditions on a multisymplectic Hamiltonian system under which a reduced premultisymplectic form is nondegenerate. A related problem is to determine those conditions under which a reduced Hamiltonian multivector field is guaranteed to be Hamiltonian. It would be interesting if such conditions could be interpreted in terms of classical field theory or Nambu mechanics.

	In the opposite direction, as observed in Remark \ref{rem:multisymplectic_reduction_remark} the hypotheses of the multisymplectic reduction theorem can be significantly weakened. It would be interesting to see how far this can be taken and what form the associated theory of generalized Hamiltonian systems takes. As a related application, it may be possible to apply these ideas in conjunction with \cite{BursztynMartinezAlbaRubio19,BursztynCavalcantiGualtieri07} to obtain a reduction theory in more general settings. It may also be interesting to investigate foliations tangent to the kernel distribution of an auxiliary component $\eta\in\Omega^{k-1}(M)$ of a split moment map $\mu=\nu\wedge\eta$.

	\item \textbf{Geometry of split moment maps.} Many results of symplectic geometry involve the comparison of the values of the moment map $\mu:M\to\g^*$ at different points of the $M$. For example, the convexity of the image of the moment map \cite{Atiyah82,GuilleminSternberg82a}, the classification of toric symplectic manifolds \cite{Delzant88}, and the Kirwan surjectivity theorem \cite{Kirwan84}. Split Hamiltonian $G$-spaces $(M,\omega,G,\nu\wedge\eta)$ provides a natural setting for the reinterpretation of these statements in the $k$-plectic setting: specifically, in terms of the function $\nu:M\to\g^*$ under various conditions on the auxiliary form $\eta\in\Omega^{k-1}(M)$.

	\item \textbf{Homotopy moment maps and weak moment maps.} Invoking the theory of $L_\infty$-algebras, Callies, Fr\'{e}gier, Rogers, and Zambon have introduced the \emph{homotopy moment map} \cite{CalliesFregierRogersZambon16}. Employing the framework of Lie algebra cohomology, Madsen and Swann have developed the \emph{multi-moment map} \cite{MadsenSwann12,MadsenSwann13}. Generalizing both of these, Herman has introduced the \emph{weak (homotopy) moment map} \cite{Herman18b}. The homotopy moment map and the weak moment map each refine our construction of the multisymplectic comoment map with the addition of a family of functions involving the Lie algebra of the acting group $G$ and the space of differential forms on the underlying manifold $M$. It would be interesting to explore the interaction between these more nuanced approaches to the moment map and the framework of multisymplectic reduction.
	
	\item \textbf{Infinite-dimensional symplectic geometry.} Multisymplectic manifolds first arose in classical field theory through the \emph{multimomentum space}, a finite-dimensional bundle the sections of which form an infinite-dimensional \emph{space of fields}. The solutions of the field equations often possess a canonical symplectic structure \cite{Helein12}. In this setting, the physical situation is modeled by a finite-dimensional multisymplectic and an infinite-dimensional symplectic Hamiltonian $G$-spaces. It would be interesting to determine the extent to which the multisymplectic and symplectic reduction procedures yield equivalent reduced spaces.

	It may also be interesting to investigate the relation between certain infinite-dimensional symplectic manifolds and finite-dimensional multisymplectic manifolds more generally. The case of a symplectic fibration equipped with a volume form on its base, as exhibited in Example \ref{eg:split_Hamiltonian_systems} part ii., and the induced symplectic structure on the associated space of sections presents a natural domain of applications. Interactions between multisymplectic geometry, hydrodynamics, and knot theory have appeared in \cite{MitiSpera19}, which may suggest further interesting applications. General introductions to infinite-dimensional symplectic geometry are presented in \cite{Marsden67,ChernoffMarsden74}. 

	\item \textbf{Quantization.} Approaches to multisymplectic quantization have been advanced by Barron and Seralejahi \cite{BarronSerajelahi17}, Barron and Shafiee \cite{BarronShafiee19}, de Bellis, Samann, and Szabo \cite{DeBellisSamannSzabo10,DeBellisSamannSzabo11}, H{\'e}lein \cite[Section 3]{Helein12}, Rogers \cite{Rogers13}, and Serajelahi \cite{Serajelahi15}. It would be interesting to understand the interaction between these theories and multisymplectic reduction. One natural question, for example, is to determine conditions under which a suitably reformulated Guillemin--Sternberg $[Q,R]=0$ conjecture \cite{GuilleminSternberg82} obtains in the multisymplectic setting.
\end{enumerate}


\subsection*{Acknowledgments}
The author would like to thank Manuel de Le{\'o}n for helpful comments on an early draft of this paper, and Takuya Sakasai for suggesting the topic of generalized Duistermaat--Heckman theorems. The author would also like to acknowledge the support of the East China Normal University, the China Postdoctoral Science Foundation, and the Euler International Mathematical Institute in Saint Petersburg. This work is supported by the Ministry of Science and Higher Education of the Russian Federation, agreement N\textsuperscript{\underline{o}}\hspace{-1.2pt} 075--15--2019--1619.


\bibliography{multisymplectic}
\bibliographystyle{abbrv}

\par
\medskip
\begin{tabular}{@{}l@{}}
	\textsc{Saint Petersburg State University}, \textit{and} \\
	\textsc{Leonhard Euler International Mathematical Institute in Saint Petersburg,} \\
	\textsc{14th Line 29B, Vasilyevsky Island, Saint Petersburg, 199178, Russia} \\[1.5pt]
	\textit{E-mail address}: \texttt{c.blacker@eimi.ru}
\end{tabular}

\end{document}